\documentclass[12pt, oneside, psamsfonts]{amsart}

\newif\ifPDF
\ifx\pdfoutput\undefined\PDFfalse
\else \ifnum \pdfoutput > 0 \PDFtrue
        \else \PDFfalse
        \fi
\fi

\usepackage[centertags]{amsmath}
\usepackage{amsfonts}
\usepackage{mathrsfs}
\usepackage{textcomp}
\usepackage{amssymb}
\usepackage{amsthm}
\usepackage{newlfont}
\usepackage[all]{xy}

\ifPDF
  \usepackage[pdftex]{xcolor, graphicx}
  \usepackage[pdftex, bookmarks, colorlinks]{hyperref}
  \hypersetup{colorlinks=false}

\else
  \usepackage{color}
  \usepackage[dvips]{graphicx}
  \usepackage[dvips]{hyperref}
\fi

\usepackage{tkz-graph}
\tikzset{EdgeStyle/.style = {->}}
\tikzset{LabelStyle/.style= {fill=yellow}}
\usetikzlibrary{shapes,snakes,calendar,matrix,backgrounds,folding}

\usepackage[scale=0.8]{geometry}

\newtheorem{thm}{Theorem}[section]
\newtheorem{cor}[thm]{Corollary}
\newtheorem{lem}[thm]{Lemma}
\newtheorem{prop}[thm]{Proposition}
\theoremstyle{definition}
\newtheorem{defn}[thm]{Definition}
\newtheorem{rem}[thm]{Remark}
\newtheorem{example}[thm]{Example}
\numberwithin{equation}{section}

\theoremstyle{definition}

\newcommand{\norm}[1]{\left\Vert#1\right\Vert}
\newcommand{\abs}[1]{\left\vert#1\right\vert}

\newcommand{\Int}{\mathbb Z}
\newcommand{\Comp}{\mathbb C}
\newcommand{\Ratn}{\mathbb Q}
\newcommand{\eps}{\varepsilon}

\newcommand{\Kzero}{\mathrm{K}_0}
\newcommand{\Kone}{\mathrm{K}_1}

\newcommand{\N}{\mathbb N}
\newcommand{\Z}{\mathbb Z}

\newcommand{\e}{\varepsilon}
\newcommand{\ov}{\overline}

\begin{document}

\title[C*-algebras of a homeomorphism on a Cantor set]{C*-algebras of a Cantor
system with finitely many minimal subsets: structures, K-theories, and the index map}

\author{Sergey Bezuglyi}
\address{University of Iowa, Iowa City, Iowa, USA}
\email{sergii-bezuglyi@uiowa.edu}

\author{Zhuang Niu}
\address{Department of Mathematics, University of Wyoming, Laramie, Wyoming, USA,~~82071.}
\email{zniu@uwyo.edu}

\author{Wei Sun}
\address{Department of Mathematics, Eastern China Normal University, Shanghai, China.}
\email{wsun@math.ecnu.edu.cn}

\thanks{}
\keywords{}
\date{\today}
\dedicatory{}
\commby{}


\begin{abstract}
We study  homeomorphisms of a Cantor set with $k$ ($k < +\infty$) minimal 
invariant closed (but not open) subsets; we also study  crossed product 
C*-algebras associated to these Cantor systems and their certain orbit-cut 
sub-C*-algebras. In the case that $k\geq 2$, the crossed product C*-algebra is stably finite,  
has stable rank 2, and has real rank zero if in addition $(X, \sigma)$ is aperiodic. 
The image of the index map is connected to certain directed graphs 
arising from the Bratteli-Vershik-Kakutani model of the Cantor system. 
Using this, it is shown that the ideal of the Bratteli diagram (of the Bratteli-Vershik-Kakutani model) must have at least $k$ vertices at each level, and the image of the index map must consist infinitesimals.

\end{abstract}

\maketitle


\setcounter{tocdepth}{1}
\tableofcontents


\section{Introduction}

The paper is devoted to the study of interactions between topological
dynamical systems on a Cantor set and C*-algebras. 
There are several constructions of operator algebras arising from dynamical systems in the frameworks of
measurable or topological dynamics, and the most important one is the concept of
crossed product operator algebras which is originated in the work of von Neumann. 
It is fascinating to see how dynamical properties reveal themselves in the corresponding
operator algebras. On the other hand, ideas and methods used in operator
algebras can give new approaches and information for the study of dynamical 
systems. We mention here that Bratteli diagrams (one of the central objects 
of our paper) are proved to be an effective tool in dynamics. 

A topological space is called \textit{Cantor} if it is compact, metrizable,
 totally disconnected, and has no isolated points. Any two Cantor sets are 
 homeomorphic. A \textit{Cantor system }$(X, \sigma)$ consists of a
  Cantor set $X$ and a homeomorphism $\sigma: X\to X$. A closed subset 
$Y\subseteq X$ is said to be invariant if $Y=\sigma(Y)$.
It follows from the Zorn's Lemma that any system $(X, \sigma)$ has 
minimal nonempty closed invariant subsets
 (they are called \textit{minimal components}). 
If $X$ is the only nonempty closed invariant subset, the system 
$(X, \sigma)$ is said to be \textit{minimal}. This means, in other words,
 that, for every $x \in X$, the $\sigma$-orbit of $x$ in dense in $X$. 
 In this paper, we will 
focus on \textit{nonminimal homeomorphisms} of a Cantor set. 
Our primary interest is the 
case when a homeomorphism $\sigma$ has  finitely many minimal components 
 $Y_1, Y_2, ..., Y_k$. These subsets are disjoint with necessity.
 There are several natural classes of 
 homeomorphisms which have this property, e.g., interval exchange 
 transformations (treated as Cantor  dynamical systems), aperiodic 
 substitutions, etc., see \cite{BezuglyiKwiatkowskiMedynets2009, 
 BezuglyiKwiatkowskiMedynetsSolomyak2013, CornfeldFominSinai1982}.

If some of the minimal component, say $Y_i$, is open, then both $Y_i$ and $X\setminus Y_i$ are again Cantor sets (they are compact, metrizable, totally disconnected, and have no isolated points). Since $Y_i$ and $X\setminus Y_i$ are invariant under $\sigma$, the whole system $(X, \sigma)$ can be decomposed into the Cantor systems $(Y_i, \sigma)$ and $(X\setminus Y_i, \sigma)$, where $(Y_i, \sigma)$ is minimal, and$(X\setminus Y_i, \sigma)$ has $k-1$ minimal components. 

Thus, in the paper, we will assume that $(X, \sigma)$ is  \textit{indecomposable}, i.e., none of $Y_i$ are open. 
Our goal is to  study 
C*-algebras associated to $(X, \sigma)$. For this, we use the 
existence of a \textit{Bratteli-Vershik-Kakutani model} of $(X, \sigma)$ 
constructed by a  sequence of Kakutani-Rokhlin partitions. Our main 
interest  is on the properties of the C*-algebras associated
to a Cantor system  $(X, \sigma)$ with finitely many minimal sets. 
In this paper, we apply the methods developed in the K-theory and Bratteli diagrams. Note that the case of 
minimal homeomorphisms was extensively  studied in earlier papers (see, 
in particular,   \cite{Davidson-book, GPS-Cantor, HPS-Cantor, Poon-PAMS-89, 
Poon-Rocky, Put-PJM, Putnam-AT} where various aspects of 
the theory of C*-algebras corresponding  to minimal 
homeomorphisms  are  discussed).

It is well known that if $(X, \sigma)$ is minimal, then the crossed product C*-algebra $\mathrm{C}(X)
\rtimes\Int$ is isomorphic to an inductive limit of circle algebras (A$\mathbb{T}$ algebra); see \cite{Put-PJM}. This is still true if there is only one minimal component. However, once there are at least two minimal components, the C*-algebra $\mathrm{C}(X)\rtimes\Int$ is not an A$\mathbb{T}$ algebra anymore. In fact, in this case, the C*-algebra $\mathrm{C}(X)\rtimes\Int$
 is a \textit{stably finite C*-algebra with stable rank $2$}; moreover, it has {\em real rank zero} if $(X, \sigma)$ is aperiodic. (See Corollary \ref{2-0-fin}; also see \cite{Poon-PAMS-89}.)


In the case of minimal dynamical systems, it is also well known that each minimal Cantor system $(X, \sigma)$ has a Bratteli-Vershik-Kakutani model; see \cite{HPS-Cantor}. Moreover, the (unordered) Bratteli diagram appears in the Bratteli-Vershik-Kakutani model must be simple, and any (non-elementary) simple Bratteli diagram can be ordered to have a continuous Bratteli-Vershik map.

In the case of Cantor systems with $k$ minimal components, it turns out that any such a system still has a Bratteli-Vershik-Kakutani model (see Section \ref{Can-BD}); and very naturally, the (unordered) Bratteli diagram appears in the Bratteli-Vershik-Kakutani model is $k$-simple (Definition \ref{defn-unordered-BD}); or roughly speaking, it has $k$ sub-diagrams which are simple, and each of them corresponds to one of the minimal components.

However, unlike the simple case, to assign an order on a given (unordered) $k$-simple Bratteli diagram with $k\geq 2$ so that the Bratteli-Vershik map is continuous is somehow restrictive (see Definition \ref{defn-BD}), and is even impossible on some ($k$-simple) Bratteli diagrams. Therefore, not all $k$-simple Bratteli diagram can arise from a Cantor system. For more details on orders on Bratteli diagrams, see \cite{BezuglyiKwiatkowskiYassawi2014, BezuglyiYasaswi2017, GPS-Cantor, HPS-Cantor,
 JanssenQuasYassawi2017,  Medynets2006}.


It turns out that these constrains on the Bratteli diagrams can be connected to the crossed product algebra $\mathrm{C}(X)\rtimes\Int$ and the associated C*-algebra extension
\begin{equation*}
\xymatrix{
0\ar[r]&\mathrm{C}_0(X\setminus\bigcup_iY_i)
\rtimes_\sigma\Int\ar[r]&\mathrm{C}(X)\rtimes_\sigma\Int\ar[r]&
\bigoplus_i\mathrm{C}(Y_i)\rtimes_\sigma\Int\ar[r]&0
}.
\end{equation*}
The index map of the extension above, $$\Kone(\bigoplus_i\mathrm{C}(Y_i)\rtimes_\sigma\Int) \to \Kzero(\mathrm{C}_0(X\setminus\bigcup_{i=1}^k Y_k)\rtimes\Int),$$
turns out to have the image isomorphic to $\Int^{k-1}$ (in particular, the index map is nonzero if 
$k\geq 2$).

To connect the index map to Bratteli diagrams, we 
introduce a sequence of finite directed graphs (transition graphs, see Definition 
 \ref{definition-TD}) for the ordered Bratteli diagram which represents the given Cantor system $(X, 
 \sigma)$. Then the index map can be recovered from these 
 transition graphs (Theorem \ref{index-diag}). Using this connection as a bridge, 
we show that  elements in the image of the index map  must be  
\textit{infinitesimals} (Corollary \ref{index-inf}). The transition graphs are also
  used in Section \ref{EHS} to characterize the (unordered) Bratteli 
  diagrams which appear in the Bratteli-Vershik-Kakutani construction (Theorem \ref{realization-TD} and Theorem 
  \ref{realization-IF} ). Relative results on graphs associated to Bratteli 
  diagrams also can be found in \cite{BezuglyiKwiatkowskiYassawi2014, BezuglyiYasaswi2017}.
  
  The paper is organized as follows. In Section \ref{sect prelim}, we 
include several definitions and facts from K-theory and 
Cantor dynamics that are used in the main part of the paper. In Section 
\ref{sect C.s. finitely many}, we discuss basic properties of the crossed
product C*-algebras, generated by homeomorphisms $\sigma$
with finitely many minimal components. Using the Pimsner-Voiculescu six-term exact sequence, one shows that (in Subsection \ref{subsect index map}) 
the index map is nonzero if the Cantor system $(X, \sigma)$ contains
more than one minimal components, and hence in this case, the crossed product C*-algebra cannot be A$\mathbb T$, in contrast to the case of minimal Cantor systems. The K-theory of the Putnam's orbit cutting algebra $A_Y$, where $Y\subseteq X$ is a closed set with non-empty intersection with each $Y_i$, $i=1, ..., k$, is considered in Section \ref{sect A_Y} (for instance, $Y$ can be $\bigcup_i Y_i$).
A dynamical system description of $\Kzero(A_Y)$ is given in Theorem \ref{k0-iso}.

In Section \ref{BV-model}, $k$-simple ordered Bratteli diagrams are 
introduced. They are Bratteli diagrams with $k$ simple quotients, but 
with more compatibility conditions on the order structure (see Definition 
\ref{defn-BD}). With these conditions, the Bratteli-Vershik map is then 
continuous, and the induced Cantor system is indecomposable and has 
$k$ minimal components. Then, in Section \ref{Can-BD}, one also shows
 that these $k$-simple ordered Bratteli diagrams (with the Vershik maps) 
 are exactly the Bratteli-Vershik-Kakutani models for arbitrary indecomposable Cantor 
 systems with $k$ minimal components.

The compatibility conditions on the $k$-simple ordered Bratteli diagram 
actually are very rigid on the underlining (unordered) Bratteli diagram, 
and it has a deep connection to the index map of the C*-algebra extension 
associated with the Cantor system. In Section \ref{tran-graph}, we build 
this bridge by \textit{transition graphs} (Definition \ref{definition-TD}). It is a 
sequence of finite directed graphs produced from the given $k$-simple 
ordered Bratteli diagram, and it actually provides a combinatorial 
description of the index map (Theorem \ref{index-diag}). As consequences, 
when $k\geq 2$, the $\Kzero$-group of the canonical AF ideal of 
$\mathrm{C}(X)\rtimes\Int$ must contain infinitesimal elements, it has rational rank at least $k$, and 
the C*-algebra $\mathrm{C}(X)\rtimes\Int$ is stably finite. Together with Corollary 2.6 of 
\cite{Poon-PAMS-89}, this provides natural examples of stably finite C*-algebra with stable 
rank $2$ and real rank zero.

Transition graphs are used again in Section \ref{EHS} to describe which unordered Bratteli 
diagram can carry a Cantor system with $k$ minimal components. With 
help of Euler walk, this question can be answered in a special case 
(Theorem \ref{realization-IF}).

In Section \ref{stable-alg}, we consider Cantor systems with only 
one minimal component. In this case, the crossed product C*-algebra is 
an A$\mathbb T$ algebra, and hence the $\Kzero$-group is a (not 
necessarily simple) dimension group. We study the connection between 
its order structure and the boundedness of the invariant measures 
concentrated on the complement of the minimal components.

In Section \ref{Chain}, we focus on topological properties of Cantor
dynamical systems with finitely many minimal components. We discuss the 
notions of \textit{chain transitive} and \textit{moving} homeomorphisms
(the later was introduced in \cite{BezuglyiDooleyKwiatkowski2006}).
We give a necessary and sufficient condition for a non-minimal 
homeomorphism to be chain transitive, see Theorem \ref{thm chain 
transitive}. 
In particular, we show that any $k$-minimal homeomorphism is 
chain transitive (Corollary \ref{cor k-minimal homeo}).  

\subsubsection*{Acknowledgements} Part of the work were carried out
 in 2010-2011 when the third named author was a Postdoctoral Fellow at the 
 Memorial University of Newfoundland; the second named author and the third 
 name author thank a start-up grant of the Memorial University of 
 Newfoundland for the support. 
 The work of the second named author is  partially supported by an NSF grant 
 (DMS-1800882). The first named author thanks the University of Wyoming 
 for the warm hospitality during his visits. 
 
\section{Preliminaries and Notation}\label{sect prelim}

\subsection{Cantor systems}
By a Cantor dynamical system $(X, \sigma)$, we mean a Cantor set $X$ together with a homeomorphism $\sigma: X\to X$. A closed set $Y\subseteq X$ is said to be invariant if $\sigma(Y) = Y$, and a closed invariant set is said to be {\em minimal} if it is non-empty and it does not contain any closed invariant subsets other than itself and $\O$. By the Zorn's Lemma and the compactness of $X$, minimal invariant subsets always exist, and let us call them {\em minimal components.} In this paper, we only consider Cantor systems with finitely many minimal components, unless
other properties of $\sigma$ are  explicitly specified. A Cantor system is said to be {\em aperiodic} if $X$ does not contain any periodic points of $\sigma$.

\subsection{Ordered Bratteli diagrams}\label{subsect OBD}
Let us recall some definitions and notation on Bratteli diagrams which are  crucial in Cantor dynamics for constructing   models for given transformations. 
For more details, we refer to \cite{BezuglyiKarpel2016, DownarowiczKarpel2018, Durand2010} where various combinatorial and dynamical properties of simple and non-simple Bratteli diagrams are discussed.

\begin{defn}\label{def BD}
 A {\it Bratteli diagram} is an infinite graph $B=(V,E)$ such that the vertex
set $V =\bigcup_{n\geq 0}V^n$ and the edge set $E= 
\bigcup_{n\geq 0}E^n$ are partitioned into disjoint subsets $V^n$ and 
$E^n$ where

(i) $V^0=\{v_0\}$ is a single point;

(ii) $V^n$ and $E^n$ are finite sets, $\forall n \geq 0$;

(iii) there exist $r : E \to V$ (range map $r$) and $s : E \to V$ 
(source map $s$), both from $E$ to $V$, such that $r(E^n)= V^{n+1}$, 
$s(E^n)= V^{n}$ (in particular, $s^{-1}(v)\neq\O$ and 
$r^{-1}(v')\neq\O$ for all $v\in V$ and $v'\in V\setminus V_0$).
\end{defn}

The structure of every Bratteli diagram $B$ is completely determined 
by the sequence of incidence matrices $(F_n)$. By definition, the 
\textit{incidence matrix} $F_n$ has entries
$$
f^{(n)}_{v,w} =|\{ e \in E_n : s(e) = w, r(e) =v\}|,\ v\in V_{n+1}, w 
\in V_n.
$$

A Bratteli diagram $B$ is called \textit{stationary} if $F_n = F_1$ for all 
$n$, and $B$ is of \textit{finite rank} if there exist $d\in \mathbb N$ 
such that $|V_n| \leq d$ for all $n$.

In what follows, we will constantly use the telescoping procedure. A 
telescoped Bratteli diagram  preserves all properties of the initial diagram
so that it does not change its dynamical properties. 

\begin{defn} \label{def telescoping}
Let $B$ be a Bratteli diagram, and $n_0 = 0 <n_1<n_2 < \cdots$ be a 
strictly increasing sequence of integers. The {\em telescoping of $B$ to 
$(n_k)$} is the Bratteli diagram $B'$, with $i$-level vertex set $(V')^i$
 to be $V^{n_i}$ and the edges between $(V')^i$ and $(V')^{i+1}$ to 
 be finite paths of the Bratteli diagram $B$ between $n_i$-level vertices
 and $n_{i+1}$-level vertices.
\end{defn}

For a Bratteli diagram $B$, the {\it tail (cofinal) equivalence}  relation
$\mathcal E$ on the path space $X_B$ is defined as follows: 
$x\mathcal E y$ if
there exists $m\in \N$ such that  $x_n=y_n$ for all $n\geq m$, where 
$x = (x_n)$, $y= (y_n)$.

A Bratteli diagram $B$ is called \textit{aperiodic} if  every 
$\mathcal  E$-orbit is countably infinite.

\begin{lem}\label{aperiodic-property}
Every aperiodic Bratteli diagram $B$ can be telescoped to a diagram $B'$
with the property: $|r^{-1}(v)| \geq 2,\ v \in V\setminus V^0 $ and
$|s^{-1}(v)| \geq 2, \ v\in V \setminus V^0$.
\end{lem}

\begin{rem}
Given an aperiodic dynamical system $(X,\sigma)$, a Bratteli diagram
is constructed by a sequence of refining Kakutani-Rokhlin partitions 
generated by $(X,\sigma)$ (for details, see \cite{HPS-Cantor, 
Medynets2006}).
The $n$-th level of the diagram, $(V^n, E^n)$, corresponds to the $n$-th
Kakutani-Rokhlin partition,  and the cardinality of the set 
$E(v_0, v)$ of all finite paths  between the top
$v_0$ and a vertex $v \in V^n$ is the height of
the $\sigma$-tower labeled by the symbol $v$ from that partition.
We will give more details of this construction in Section 
\ref{Can-BD}. 
\end{rem}

\begin{defn}\label{order_definition} A Bratteli diagram $B=(V^*,E)$ 
is called {\it ordered}
if there is a linear order ``$>$" on every set  $r^{-1}(v)$, $v\in
\bigcup_{n\ge 1} V_n$. We also use $>$ to denote the corresponding 
partial order on $E$ and write $(B, >)$ when we consider $B$ with 
the order $>$. 
\end{defn}

Every order $>$  defines the  \textit{lexicographic}
ordering on the set $E(k,l)$  of finite paths between
vertices of levels $V^k$ and $V^l$: it is said that 
 $$
 (e_{k+1},...,e_l) > (f_{k+1},...,f_l)
 $$
if and only if there is $i$ with $k+1\le i\le l$ such that $e_j=f_j$ for $i<j\le l$, 
and $e_i> f_i$.
It follows that, given the order $>$, any two paths from $E(v_0, v)$
are comparable with respect to the lexicographic ordering generated 
by $>$.  If two infinite paths are tail equivalent, and agree 
from the vertex $v$ onwards, then we can compare them by comparing 
their initial segments in $E(v_0,v)$. Thus $>$ defines a partial order 
on $X_B$, where two infinite paths are comparable if and only if they 
are tail equivalent.

With every Bratteli diagram $B$, one can associate the 
\textit{dimension group}
 ${\mathrm{K}}_B$. This is defined as the direct limit of groups 
$\mathbb Z^{|V_n|}$ with a sequence of positive homomorphisms 
generated by incidence matrices $F_n$:
$$
{\mathrm{K}}_B = \lim_{n \to \infty} (\mathbb Z^{|V_n|} \stackrel{F_n} 
\longrightarrow \mathbb Z^{|V_{n+1}|}). 
$$  
Then ${\mathrm{K}}_B$ is an abelian partially order group with 
 the distinguished order unit corresponding to  $V^0$. These groups
 play a prominent role in classification of AF algebras and Bratteli 
 diagrams, see, e.g.,  \cite{Br-AF, Elliott-AF, GPS-Cantor}. 
 
In particular, it is well known that, for Bratteli diagrams $B_1$ and $B_2$, 
if ${\mathrm{K}}_{B_1}\cong {\mathrm{K}}_{B_2}$ (as order-unit groups), then there is a Bratteli diagram $B$ such
 that both $B_1$ and $B_2$ can be obtained from $B$ by the 
 telescoping operations.

\subsection{Crossed product C*-algebras and K-theory} 
 Consider a compact Hausdorff space $X$ and consider a homeomorphism $\sigma: X \to X$. Let $\mathrm{C}(X)$ denote the C*-algebra  of complex-valued continuous  functions on $X$. Then the homeomorphism $\sigma$ induces an automorphism of  $\mathrm{C}(X)$ by $f\mapsto f\circ\sigma^{-1} = \sigma(f)$, which is 
 still denoted by $\sigma$. The \textit{crossed product} C*-algebra 
 $\textrm{C}(X)\rtimes_\sigma\Int$ is defined to be the universal 
 C*-algebra generated by $\mathrm{C}(X)$ and a canonical unitary 
 $u$ satisfying 
 $$u^*fu=\sigma(f),\quad f\in\mathrm{C}(X).
 $$ 
In other words, the crossed product C*-algebra is defined by 
$$
\textrm{C}(X)\rtimes_\sigma\Int:=\textrm{C*}\{f, u;\ f\in
\mathrm{C}(X), uu^*=u^*u=1, u^*fu=f\circ\sigma^{-1}\}.
$$ 
 This concept is defined and studied in many books on operator algebras, 
 see, for instance, Chapter VIII of \cite{Davidson-book}.
 
Applying the Pimsner-Voiculescu six-term exact sequence to $\mathrm{C}(X) \rtimes_\sigma \Int$, 
one obtains:
\begin{equation}\label{eq PV seq}
\xymatrix{
\Kzero(\mathrm{C}(X))\ar[r]^{1-[\sigma]_0}&\Kzero(\mathrm{C}(X)
\ar[r])&
\Kzero(\mathrm{C}(X)\rtimes_\sigma\Int)\ar[d]\\
\Kone(\mathrm{C}(X)\rtimes_\sigma\Int)\ar[u]&\Kone(\mathrm{C}(X))
\ar[l]&
\Kone(\mathrm{C}(X))\ar[l]^{1-[\sigma]_1})},
\end{equation}
where $\Kzero(A)$ and $\Kone(A)$ are $\Kzero$-group and 
$\Kone$-group 
of a C*-algebra $A$ respectively.  In (\ref{eq PV seq}), 
$[\sigma]_0$ and $[\sigma]_1$ denote the maps between 
$\Kzero$-groups and $\Kone$-groups induced by $\sigma$, respectively. 

In the case that $X$ is a Cantor set, one has 
$\Kone(\mathrm{C}(X))=\{0\}$; hence, it follows from (\ref{eq PV seq})
that
\begin{equation}\label{equ_K0}
\Kzero(\mathrm{C}(X)\rtimes_\sigma\Int)=
\Kzero(\textrm{C}(X))/(1-[\sigma]_0)(\Kzero(\textrm{C}(X))),
\end{equation} 
and the group
\begin{equation}\label{equ_K1}
\Kone(\mathrm{C}(X)\rtimes_\sigma\Int)=\ker(1-[\sigma]_0)
\end{equation}
consists of $\sigma$-invariant functions. 

A closed two-sided \textit{ideal} of a C*-algebra $A$ is a sub-C*-algebra $I\subseteq A$ such that $IA\subseteq I$ 
and $AI\subseteq I$. If $Y\subseteq X$ is a closed subset, then 
$\mathrm{C}_0(X\setminus Y)$ is an ideal of $\mathrm{C}(X)$. If, 
moreover, $Y$ is invariant (i.e., $\sigma(Y) = Y$), then 
$\mathrm{C}_0(X\setminus Y)$ is a $\sigma$-invariant ideal of 
$\mathrm{C}(X)$, and therefore 
$\mathrm{C}_0(X\setminus Y)\rtimes_{\sigma}\Int$ is an ideal of 
$\mathrm{C}(X)\rtimes_{\sigma}\Int$ with quotient canonically 
isomorphic to $\mathrm{C}(Y) \rtimes_{\sigma} \Int$. That is, one has a short exact sequence of C*-algebras:
\begin{equation*}
\xymatrix{
0\ar[r]&\mathrm{C}_0(X\setminus Y)
\rtimes_\sigma\Int\ar[r]&\mathrm{C}(X)\rtimes_\sigma\Int\ar[r]&
\mathrm{C}(Y)\rtimes_\sigma\Int\ar[r]&0
}.
\end{equation*}

In general, for any C*-algebra extension 
$$
\xymatrix{
0 \ar[r] & I \ar[r] & A \ar[r] & A/I \ar[r] & 0,
}
$$
one has the six-term exact sequence
$$
\xymatrix{
\Kzero(I)\ar[r] & \Kzero(A)\ar[r] & \Kzero(A/I) \ar[d] \\
\Kone(A/I) \ar[u] & \Kone(A) \ar[l] & \Kone(I) \ar[l].
}
$$
The map $\Kone(A/I) \to \Kzero(I)$ is called the \textit{index map},
 and the map $\Kzero(A/I) \to \Kone(I)$ is called the \textit{exponential
  map}.

For more information about the K-theory of a C*-algebra, see, for instance, 
\cite{RLL-ktheory}.

\section{Cantor system with finitely many minimal subsets}
\label{sect C.s. finitely many}

\subsection{Crossed product C*-algebra associated to $\sigma$}
Recall that, for a topological dynamical system $(X, \sigma)$
($X$ is not necessarily a Cantor set), a closed 
subset $Y$ is minimal if $Y$ is a closed invariant nonempty subset and 
$Y$ is minimal among these subsets. Minimal subsets always exist, and 
each pair of them are disjoint.

Let $(X, \sigma)$ be a topological dynamical system with only $k$ 
minimal subsets $Y_1, ..., Y_k$. We then have the short exact sequence
\begin{equation}\label{exact}
\xymatrix{
0\ar[r]&\mathrm{C}_0(X\setminus\bigcup_iY_i)
\rtimes_\sigma\Int\ar[r]&\mathrm{C}(X)\rtimes_\sigma\Int\ar[r]&
\bigoplus_i\mathrm{C}(Y_i)\rtimes_\sigma\Int\ar[r]&0
}.
\end{equation}


\begin{lem}\label{AF-cond}
Let $Y_1,. ... , Y_k$ be all the minimal subsets of $(X, \sigma)$.
Suppose that $U\subseteq X$ is an open set such that $U \cap Y_i \neq \O$, 
$i=1, 2, ..., k$. Then $\bigcup_{j=-\infty}^\infty\sigma^{j}(U)=X$.
\end{lem}
\begin{proof}
We first note that if an open set $V$ intersects a minimal set $Y$, then
the orbit of $V$ contains $Y$.  Consider the open invariant subset 
$$Z:=\bigcup_{j=-\infty}^\infty\sigma^{j}(U).
$$ 
If $Z$ were a proper subset of $X$, then $X\setminus Z$ 
would be a nonempty invariant closed subset of $X$, and hence it must
 contains a minimal subset $Y$. But this minimal subset is disjoint with 
every set $Y_1, ... , Y_k$. This contradicts the assumption that $Y_1, Y_2, ..., Y_k$ are all the minimal componenents. Therefore, $Z=X$, as desired.
\end{proof}

\begin{rem} 
It follows from Lemma \ref{AF-cond} that for any open subset 
$U\supseteq\bigcup_i Y_i$, one has that $\bigcup_{j=-\infty}^\infty
\sigma^{j}(U)=X$.
\end{rem}

\begin{lem}\label{non-int}
Under the assumption that all proper minimal components 
for $(X, \sigma)$ are not clopen sets,
every minimal compenent $Y$  has empty interior.
\end{lem}

\begin{proof}
Suppose $Y$ has a subset $U$ which is open in $X$.  Since $Y$ is minimal
 and compact, there exists $N$ such that $Y=\bigcup_{j=-N}^N 
 \sigma^j(U)$.  Hence $Y$ is open in $X$, which contradicts to the
 assumption.
\end{proof}

\begin{cor} Let $Y_1, ..., Y_k$ be minimal subsets for $(X, \sigma)$. 
If none of $Y_i$ is clopen, then the ideal $\mathrm{C}_0(X\setminus
\bigcup_iY_i)$ is essential in $\mathrm{C}(X)$.
\end{cor}

\begin{prop}
Let $\alpha$ be an action of $\Int$ on a C*-algebra $A$, and let $I$ be an
 invariant ideal. If $I$ is essential in $A$, then $I\rtimes_\alpha\Int$ is 
 essential in $A\rtimes_\alpha\Int$.
\end{prop}

\begin{proof}
Consider the conditional expectation $$\mathbb{E}: A\rtimes_\alpha\Int 
\ni a \mapsto \int_{\mathbb T}\alpha^*_t(a) dt \in A,$$ where 
$\alpha^*_t$ is the dual action of $\alpha$. Note that $\mathbb{E}$ is 
faithful.

Let $a$ be a positive element in $A\rtimes_\alpha\Int$ with 
$a(I\rtimes_\alpha\Int)=\{0\}$. Then, for any $b\in I$, one has that 
$ab=0$. Since $b\in I\subseteq A$, one also has that $\mathbb{E}(ab)=
\mathbb{E}(a)b=0$. 
Since $\mathbb{E}(a)\in A$ and $b$ is arbitrary, one has that $\mathbb{E}
(a)=0$. Since $\mathbb E$ is faithful, one has that $a=0$, as desired.
\end{proof}

\begin{cor} For $(X, \sigma)$ as above, 
if none of $Y_1, ... , Y_k$ is clopen, then the ideal $\mathrm{C}_0
(X\setminus\bigcup_iY_i)\rtimes_\sigma\Int$ is essential in 
$\mathrm{C}(X)\rtimes_\sigma\Int$.
\end{cor}

%
%



 \subsection{AF sub-C*-algebras} 
We return to the case when the compact topological space  $X$ is a Cantor 
set.  Then each minimal subset $Y_i, i = 1, ..., k$, is a Cantor set or consists
of a periodic orbit. Indeed, to see this, suppose that a minimal set,
say  $Y_1$, contains an isolated point $x$.  Then $\textrm{Orbit}(x)$ is 
an open set (relative to $Y_1$), and therefore $\overline{\textrm{Orbit}(x)}\setminus
\textrm{Orbit}(x)$ is an invariant closed subset of $Y_1$. Since $Y_1$ is 
minimal, one has that $Y_1=\textrm{Orbit}(x)$, and then a standard 
compactness argument shows that $Y_1$ consists of a periodic orbit.

Let $Y\subseteq X$ be a closed subset with $Y\cap Y_i\neq \O$, 
$i=1, ..., k$. By Lemma \ref{AF-cond}, one has that $\bigcup_{j=-\infty}^\infty\sigma^j(U)=
X$ for any clopen set $U \supseteq Y$. Then, by 
  \cite[Lemma 2.3]{Poon-Rocky}, the sub-C*-algebra 
$$A_Y:=\textrm{C*}\{g, fu;\ g, f\in\mathrm{C}(X), f|_Y=0\}\subseteq 
\textrm{C}(X)\rtimes_\sigma\Int$$ is approximately finite (AF) dimensional.

In particular, fix $y_i\in Y_i$, $i=1, 2, ..., k$, and define 
$$A_{y_1,...,y_k}:=\textrm{C*}\{g, fu;\ g, f\in \textrm{C}(X), f(y_i)=0,
 i=1,...,k\}\subseteq \textrm{C}(X)\rtimes_\sigma\Int .$$ The C*-algebra 
 $A_{y_1,...,y_k}$ is AF, and hence the ideal $\mathrm{C}_0(X\setminus
 \bigcup_iY_i)\rtimes_\sigma\Int$ is also AF, see 
 \cite[Theorem 3.1]{Ell-Her-AF}.


 \subsection{Index maps} \label{subsect index map}
The six-term exact sequence associated to \eqref{exact} is 
\begin{equation}\label{eq 6 term crss prop}
\xymatrix{
\Kzero(\mathrm{C}_0(X\setminus\bigcup_iY_i)\rtimes_\sigma\Int)\ar[r]&
\Kzero(\mathrm{C}(X)\rtimes_\sigma\Int\ar[r])&\bigoplus_i
\Kzero(\mathrm{C}(Y_i)\rtimes_\sigma\Int)\ar[d]\\
\bigoplus_i\Kone(\mathrm{C}(Y_i)\rtimes_\sigma\Int)
\ar[u]^{\mathrm{Ind}}&\Kone(\mathrm{C}(X)\rtimes_\sigma\Int\ar[l])&
\Kone(\mathrm{C}_0(X\setminus\bigcup_iY_i)\rtimes_\sigma\Int\ar[l])
\cong\{0\}
}.
\end{equation}

Since the restriction of $\sigma$ to each $Y_i$ is minimal, it follows from  
\eqref{equ_K1}  that 
$$
\Kone(\mathrm{C}(Y_i)\rtimes_\sigma\Int)=\Int,
$$ 
and it is generated by the $\Kone$-class of the canonical unitary of $\mathrm{C}(Y_i)\rtimes_\sigma\Int$.
Also note that $(X, \sigma)$ is indecomposable, that is, the only clopen 
$\sigma$-invariant subsets are $X$ and $\O$. Then, the only invariant 
$\Int$-valued continuous functions on $X$ are constant functions, and then, 
by \eqref{equ_K1} again, we obtain that
  $$\Kone(\mathrm{C}(X)\rtimes_\sigma\Int)=\Int,$$
and it is generated by the $\Kone$-class of the canonical unitary of  $\mathrm{C}(X)\rtimes_\sigma\Int$.

\begin{thm}\label{index0} 
If a Cantor system $(X, \sigma)$ is indecomposable and has $k$ minimal components, then the the image of the index map $\mathrm{Ind}$  is isomorphic to $\Int^{k-1}$; in particular, index map is nonzero if $k\geq 2$.
\end{thm}
\begin{proof}
Note that the canonical unitary of $\mathrm{C}(X)\rtimes_\sigma\Int$ is sent to the canonical unitaries of $\mathrm{C}(Y_i)\rtimes_\sigma\Int$, $i=1, 2, ..., k$. Therefore, the index map 
$\Kone(\mathrm{C}(X)\rtimes_\sigma\Int)\to \bigoplus_{i=1}^k\Kone(\mathrm{C}
(Y_i)\rtimes_\sigma\Int)$ in \eqref{eq 6 term crss prop} is given by $$\Int\ni 1\mapsto (1, ..., 1)\in \Int^k,$$
and hence
\begin{equation}\label{eq-ind}
\mathrm{Image}(\mathrm{Ind})\cong (\bigoplus_{i=1}^k    
\Kone(\mathrm{C}(Y_i)\rtimes_\sigma\Int))/\Kone(\mathrm{C}(X)\rtimes_
\sigma\Int)\cong \Int^k/\Int(1, 1, ..., 1),
\end{equation} 
which is isomorphic to $\Int^{k-1}$ and is nonzero if $k
\geq 2$.
\end{proof}

An abelian group $G$ is said to have $\Ratn$-\textit{rank} $r$ 
if the vector space $G\otimes\Ratn$ has dimension $r$ over $\Ratn$. 
Then, another consequence of Equation \eqref{eq-ind} is the
 following:
\begin{thm}\label{ind-Q}
The $\Ratn$-rank of the group $\Kzero(\mathrm{C}_0(X\setminus
\bigcup_iY_i)\rtimes_\sigma\Int)$ is  at least $k-1$.
\end{thm}

\begin{rem}
The lower bound of the $\Ratn$-rank of $\Kzero(\mathrm{C}_0(X\setminus
\bigcup_iY_i)\rtimes_\sigma\Int)$ will be be improved to $k$ in Corollary 
\ref{rank-of-ideal}.
\end{rem}


For a clopen set $U\subset X$, denote by
$\chi_U$ the characteristic  function of the set $U$; note that $\chi_U 
\in \mathrm{C}(X)$. Also note that any continuous integer-valued function on $X$ which vanishes on $\bigcup_{i=1}^k Y_i$ induces a $\Kzero$-class of the C*-algebra $\mathrm{C}_0(X\setminus \bigcup_{i=1}^k Y_i)$ and hence induces a $\Kzero$-class of the C*-algebra $\mathrm{C}_0(X\setminus \bigcup_{i=1}^k Y_i)\rtimes_\sigma\Int$ by the embedding of $\mathrm{C}_0(X\setminus \bigcup_{i=1}^k Y_i)$.

\begin{thm}\label{index}
Consider the minimal components 
$Y_i$, $i =1,..., k,$ and choose pairwise disjoint  clopen sets $U_i\supseteq Y_i$.  
Define $$d_i:=[\chi_{U_i}-\chi_{U_i}\circ\sigma^{-1}]_0\in
\Kzero(\mathrm{C}_0(X\setminus\bigcup_iY_i)\rtimes_\sigma\Int).$$ (Note that each integer-valued continuous function $\chi_{U_i}-\chi_{U_i}\circ\sigma^{-1}$ vanishes on $\bigcup_{i=1}^k Y_i$.)
Then $d_i$ is independent of the choice of $U_i$. Moreover, $d_1=0$ if 
$k=1$. If $k\geq 2$, then $$\sum_{i=1}^k d_i=0,$$
 and the sum of any proper subset of $\{d_1, d_2, ..., d_k\}$ is nonzero.
\end{thm}

\begin{proof}
Consider $$v_i:=\chi_{U_i}u+ \sum_{j\neq i}\chi_{U_j}\in \mathrm{C}(X)
\rtimes_\sigma\Int,$$ where $u$ is the canonical unitary of $\mathrm{C}(X)
\rtimes_\sigma\Int$. Then, it is a lifting 
of the canonical unitary of $\mathrm{C}(Y_i)\rtimes_\sigma\Int$. Since 
$v_iv_i^*=\sum_{j=1}^k\chi_{U_j}$ (so $v_i$ is a partial isometry) and 
$$v_i^*v_i=u^*\chi_{U_i}u+\sum_{j\neq i}\chi_{U_j}=\chi_u\circ
\sigma^{-1}+\sum_{j\neq i}\chi_{U_j},$$ 
the index of the canonical unitary of $\mathrm{C}(Y_i)\rtimes_\sigma\Int$ is 
given by 
\begin{eqnarray*}
[1-v_i^*v_i]_0-[1-v_iv_i^*]_0&=&[1-(\chi_{U_i}\circ\sigma^{-1}+\sum_{j\neq i}
\chi_{U_j})]_0-[1-(\sum_{j=1}^k\chi_{U_j})]_0\\
&=&[\chi_{U_i}-\chi_{U_i}\circ\sigma^{-1}]_0\\
&=& d_i.
\end{eqnarray*}
Therefore, $d_i$ is the image of the $\Kone$-class of the canonical unitary of $\mathrm{C}(Y_i)\rtimes_\sigma\Int$ under the index map. Then, the theorem follows from the six-term sequence 
\eqref{eq 6 term crss prop} and \eqref{eq-ind}.
\end{proof}

\begin{cor}\label{AT}
The C*-algebra $\mathrm{C}(X)\rtimes_\sigma\Int$ is an $\mathrm{A}
\mathbb T$-algebra (i.e., it is the inductive limit of $F\otimes\mathrm{C}
(\mathbb T)$, where $F$ is a finite dimensional C*-algebra and $\mathbb T$ 
is the unit circle) if and only if $k=1$.
\end{cor}

\begin{proof}
By  \cite[Theorem 5]{LM}, an extension of A$\mathbb T$ algebras is A$
\mathbb T$ if and only if the index map is zero, and by Theorem \ref{index0}, 
this holds if and only if $k=1$.
\end{proof}

\begin{rem}
For the case $k\geq 2$, we will show later (see Corollary \ref{2-0-fin}) that 
the C*-algebra $\mathrm{C}(X)\rtimes_\sigma\Int$ is stably finite, has stable 
rank $2$, and has real rank zero if the Cantor system is aperiodic.
\end{rem}

\section{The K-theory of $A_Y$}\label{sect A_Y}
We study a general Cantor system $(X, \sigma)$ in this section.
Let $Y\subseteq X$ be a closed 
subset of $X$ which  satisfies the property: for any open subset $U$
containing $Y$, one has 
$$\bigcup_{n\in\Int} \sigma^n(U)=X.
$$ 
Such sets are called \textit{basic} in \cite{Medynets2006}.
 In particular, in the case that $X$ has $k$ minimal subsets $Y_1, ..., Y_k$, 
 the set $Y$ can be any closed subset with $Y\cap Y_i\neq\O$, $i=1, ..., k$. 
 By  \cite[Theorem 2.3]{Poon-Rocky}, $A_Y$ is an AF-algebra.  We will 
 calculate the $\Kzero$-group of $A_Y$ in this section, and show that 
  it can be identified with a certain ordered group related to the 
  dynamical system $(X, \sigma)$.

First, let us recall the following result. 

\begin{prop}[Proposition 3.3 of \cite{Poon-PAMS-89}, Theorem 4.1 of \cite{Put-PJM}]
For $Y$ chosen as above, there is an exact sequence
\begin{equation}\label{k0-exact-01}
\xymatrix{
0\ar[r]&\mathrm{C}^{\sigma}(X, \Int)\ar[r]^-\alpha&\mathrm{C}(Y, \Int)
\ar[r]^-\beta&\Kzero(A_Y)\ar[r]^-{\iota_*}&\Kzero(\mathrm{C}(X)\rtimes_
\sigma\Int)\ar[r]&0},
\end{equation}
where $\mathrm{C}^{\sigma}(X, \Int)$ is the group of $\sigma$-invariant 
integer-valued continuous functions on $X$, $\alpha$ is the restriction map, 
and $$\beta(f)=[g-g\circ\sigma^{-1}]$$ for some $g\in\mathrm{C}(X, \Int)
$ with $g|_Y=f$. In particular, if $(X, \sigma)$ is indecomposable, i.e., any $
\sigma$-invariant clopen subset is trivial, then $\mathrm{C}^{\sigma}(X, 
\Int)\cong\Int$, and \eqref{k0-exact-01} is transformed to
\begin{equation}\label{k0-exact-02}
\xymatrix{
0\ar[r]&\Int\ar[r]^-\alpha&\mathrm{C}(Y, \Int)\ar[r]^-\beta&\Kzero(A_Y)
\ar[r]^-{\iota_*}&\Kzero(\mathrm{C}(X)\rtimes_\sigma\Int)\ar[r]&0}.
\end{equation}
\end{prop}

Let us define an ordered group from the dynamical system setting.

\begin{defn}
Let $(X, \sigma)$ be a Cantor system, and let $Y\subseteq X$ be a closed 
subset. Define 
$$\mathrm{K}^0_Y(X, \sigma)=\mathrm{C}(X, \Int)/\{f-f\circ\sigma^{-1};\ 
f\in\mathrm{C}(X, \Int), f|_Y=0\},$$ and set
$$\mathrm{K}^0(X, \sigma)^+=\{\bar{f};\ f\geq 0, f\in\mathrm{C}(X, \Int)
\}\subseteq \mathrm{K}^0_Y(X, \sigma).$$
\end{defn}

Note that $\mathrm{C}(X, \Int)\cong\Kzero(\mathrm{C}(X))$ as ordered 
groups (with positive cone of $\mathrm{C}(X, \Int)$ consists of positive 
functions). Then, the embedding $\mathrm{C}(X)\subseteq A_Y$ induces a 
map $\theta: \mathrm{C}(X, \Int)\to\Kzero(A_Y)$. Moreover, a direct 
calculation shows that if $f\in \mathrm{C}(X, \Int)$ with $f|_Y=0$, then $$
\theta(f-f\circ\sigma^{-1})=0.$$ Therefore, $\theta$ induces a map from $
\mathrm{K}^0_Y(X, \sigma)$ to $\Kzero(A_Y)$ which sending $\mathrm{K}
^0(X, \sigma)^+$ to the positive cone of $\Kzero(A_Y)$.

By \eqref{equ_K0}, one has 
\begin{equation*}
\xymatrix{
0\ar[r] & \{f-f\circ\sigma^{-1};\ f\in\mathrm{C}(X, \Int)\} \ar[r] & 
\mathrm{C}(X, \Int)\ar[r]^-{\iota_*} & \Kzero(\mathrm{C}(X)\rtimes_
\sigma\Int)\ar[r] & 0
},
\end{equation*}
and therefore, there is an exact sequence
\begin{equation}\label{k0-exact-03}
\xymatrix{
0\ar[r] & H \ar[r] & \mathrm{K}_Y^0(X, \sigma)\ar[r]^-{\iota_*} & 
\Kzero(\mathrm{C}(X)\rtimes_\sigma\Int)\ar[r] & 0
},
\end{equation}
where $$H:=\frac{\{f-f\circ\sigma^{-1};\ f\in\mathrm{C}(X, \Int)\}}{\{f-f
\circ\sigma^{-1};\ f\in\mathrm{C}(X, \Int), f|_Y=0\}}.$$

For any $f\in\mathrm{C}(X, \Int)$, define $\eta(f-f\circ\sigma^{-1})$ to be 
the restriction of $f$ to $Y$. This induces an isomorphism from $H$ to $
\mathrm{C}(Y, \Int)/\alpha(\mathrm{C}^\sigma(X, \Int))$, which is still 
denoted by $\eta$. 

Indeed, if $$f-f\circ\sigma^{-1}=g-g\circ\sigma^{-1},$$ then  $$f-
g=(f-g)\sigma^{-1}.$$ That is, $f-g\in\mathrm{C}^\sigma(X, \Int)$, and 
hence $\theta$ is well defined. It is also clear that $\eta$ is a bijection, and 
thus an isomorphism.

\begin{lem}
With notation as above, the following  diagram commutes:
\begin{displaymath}
\xymatrix{
0\ar[r] & H \ar[r] \ar[d]_\eta & \mathrm{K}_Y^0(X, \sigma)\ar[r]^-
{\iota_*} \ar[d]_\theta & \Kzero(\mathrm{C}(X)\rtimes_\sigma\Int)\ar[r] 
\ar@{=}[d] & 0\\
0\ar[r] & \mathrm{C}(Y, \Int)/\alpha(\mathrm{C}^{\sigma}(X, \Int))\ar[r]^-
\beta &\Kzero(A_Y)\ar[r]^-{\iota_*}&\Kzero(\mathrm{C}(X)\rtimes_\sigma
\Int)\ar[r] &0
}.
\end{displaymath}
In particular, the map $\theta$ is an isomorphism.
\end{lem}
\begin{proof}
To prove this, one only has to verify the commutativity for the first
 square. Pick any $\overline{f-f\circ\sigma^{-1}}\in H$. Then we have
$$\beta(\eta(\overline{f-f\circ\sigma^{-1}}))=\beta(\overline{f|_Y})=[f-f
\circ\sigma^{-1}],$$ and
$$\theta(\overline{f-f\circ\sigma^{-1}})=[f-f\circ\sigma^{-1}],$$ as desired.
Since the map $\eta$ is an isomorphism, by the Five-Lemma, the map 
$\theta$ is also an isomorphism.
\end{proof}

Moreover, the map $\theta$ is in fact an order isomorphism:
\begin{thm}\label{k0-iso}
The map $\theta$ induces an order isomorphism 
$$(\mathrm{K}^0_Y(X, \sigma), 
\mathrm{K}^0_Y(X, \sigma)^+, \tilde{1})\cong(\mathrm{K}_0(A_Y), 
\Kzero^+(A_Y), [1_{A_Y}]).$$
\end{thm}
\begin{proof}
We need to show that $\theta(\mathrm{K}^0_Y(X, \sigma)^+)=\Kzero^+
(A_Y)$. 
It is clear that the image of $\mathrm{K}^0_Y(X, \sigma)^+$ is in $\Kzero^+
(A_Y)$. On the other hand, using the AF-decomposition of $A_Y$, it is clear 
that any positive element in $\Kzero(A_Y)$ comes from a positive integer 
function on $X$.
\end{proof}


Assume that $(X, \sigma)$ is indecomposable. Let $W_1\subseteq W_2$ be 
closed subsets of $X$. Then one has 
\begin{equation}\label{k0-exact-04}
\xymatrix{
0\ar[r] & K \ar[r] & \mathrm{K}_{W_2}^0(X, \sigma)\ar[r] & 
\mathrm{K}_{W_1}^0(X, \sigma)\ar[r] & 0
},
\end{equation}
where $$K=\frac{\{f-f\circ\sigma^{-1};\ f\in\mathrm{C}(X, \Int), f|_{W_1}
=0\}}{\{f-f\circ\sigma^{-1};\ f\in\mathrm{C}(X, \Int),  f|_{W_2}=0\}}.$$
For any $f\in\mathrm{C}(X, \Int)$, define $$\eta(f-f\circ\sigma^{-1})=
f|_{W_2}.$$ 
Note that this map is well defined.
One also defines $$\mathrm{C}(W_2, W_1, \Int):=\{f\in \mathrm{C}(W_2, 
\Int), f|_{W_1}=0\}.$$ Then, we can prove the following statement.
\begin{lem}
If $(X, \sigma)$ is indecomposable and $W_1\neq\O$, the map $\eta$ 
induces an isomorphism $$\eta: K\to \mathrm{C}(W_2, W_1, \Int).$$
\end{lem}
\begin{proof}
Let us first check that the map $\eta$ is well defined. 
Indeed, if $$(f-f\circ\sigma^{-1})-(g-g\circ
\sigma^{-1})=h-h\circ\sigma^{-1}$$ for some $f, g, h\in\mathrm{C}(X, \Int)
$ with $f|_{W_1}=g|_{W_1}=0$ and $h|_{W_2}=0$, then $$(f-g-h)=(f-g-h)
\circ\sigma^{-1}.$$ Since $(X, \sigma)$ is indecomposable, there is no 
invariant clopen subset of $X$, and hence $f-g-h$ is a constant function. 
Since $W_1\neq\O$ and the restrictions of $f, g, h$ to $W_1$ are zero, one 
has that $$f=g+h.$$ The condition $h|_{W_2}=0$ implies that 
$$f|_{W_2}=g|_{W_2};
$$ that is, the map $\eta$ is well defined.

It is clear that $\eta$ is surjective. If $\eta(f-f\circ\sigma^{-1})=0$, then 
$f|_{W_2}=0$; that is, the map $\eta$ is also injective, and hence it is an 
isomorphism.
\end{proof}

Thus, the exact sequence \eqref{k0-exact-04} can be written as
\begin{equation}\label{k0-exact-05}
\xymatrix{
0\ar[r] &\mathrm{C}(W_2, W_1, \Int) \ar^-{\eta^{-1}}[r] & 
\mathrm{K}_{W_2}^0(X, \sigma)\ar[r] & \mathrm{K}_{W_1}^0(X, \sigma)
\ar[r] & 0
},
\end{equation}
and applying Theorem \ref{k0-iso}, we obtain the following statement.

\begin{thm}\label{exact-subsets}
For any non-empty closed subsets $W_1\subseteq W_2$ which are basic, one has 
\begin{equation}\label{k0-exact-06}
\xymatrix{
0\ar[r] &\mathrm{C}(W_2, W_1, \Int) \ar^-{\beta}[r] & \Kzero(A_{W_2})
\ar^-{\iota_*}[r] & \Kzero(A_{W_1})\ar[r] & 0
},
\end{equation}
where $\beta(f)=g-g\circ\sigma^{-1}$ for some $g\in\mathrm{C}(X, \Int)$ 
with $g|_{W_2}=f.$
\end{thm}

\section{$k$-simple Bratteli diagrams and Bratteli-Vershik models}
\label{BV-model}

In this section, we shall introduce certain ordered Bratteli diagrams which will 
be used to model Cantor systems with finitely many minimal subsets.

\subsection{$k$-simple Bratteli diagrams}

%

Let us now introduce a special class of Bratteli diagrams. This class of Bratteli 
diagrams will serve as models for Cantor systems with finitely many minimal 
componenets.

\begin{defn}\label{defn-unordered-BD}
Let $k\in \mathbb N$. A Bratteli diagram $B= (V, E)$ is said to 
be {\em $k$-simple} if for each $n\geq 1$, there are pairwise disjoint 
subsets $V^n_1, ..., V^n_k$ of $V^n$ such that
\begin{enumerate}
\item\label{cond1}  for any $1\leq i\leq k$ and any $v\in V^{n+1}_i $, one 
has that $s(r^{-1}(v))\subseteq V^{n}_i$,
\item\label{cond2} for any $1\leq i\leq k$ and any level $n$, there is $m>n$ 
such that each vertex of $V^m_i$ is connected to all vertices of $V^n_i$.
\end{enumerate}

Moreover, denote by $V^n_o= 
V^n\setminus(V^n_1\cup\cdots\cup V^n_k)$ for $n\geq 1$. Then 
\begin{enumerate} 
\item The diagram $B$ is said to be {\em strongly $k$-simple} if for any 
level $n$, there is $m>n$ such that if a vertex $v\in V^m_o$ is connected 
to some vertex of $V^n_o$, then $v$ is connected to all vertices of 
$V^n_o$. 
\item\label{Cantorset} The diagram $B$ is said to be {\em non-elementary} 
if for any $V^n_o$, there is $m>n$ such that the multiplicity of the edges 
between $V^n_o$ and $V^m_o$ is either $0$ or at least $2$.
\end{enumerate}

A dimension group $G$ is said to be (strongly) $k$-simple if $G\cong {\mathrm{K}}_B$ 
for some (strongly) $k$-simple Bratteli diagram.
\end{defn}

\begin{rem} (1) To clarify the meaning of Definition \ref{defn-unordered-BD}, we remark that:

\begin{itemize}
\item[(a)] The diagram $B$ consists of $k$ many \textit{simple} 
sub-diagrams
 $B_i$, $i = 1,... ,k,$ constructed on the sequence of vertices $(V^n_i)$; there 
 are no edges connecting different sub-diagrams. The part of the diagram 
 whose infinite paths eventually go through vertices of $V^n_o$ for
  sufficiently large $n$ constitutes an open invariant set which does not 
  contain minimal subsets.
\item[(b)] Without loss of generality, we can assume that $V_o^n \neq \O$;
otherwise the corresponding Bratteli-Vershik system would be decomposable.
\item [(c)] The part of the diagram $B$ defined by $(V^n_o)$ induces
 an ideal of the corresponding AF-algebra.   Strongly $k$-simple diagrams
  correspond to the case that the ideal is simple.
\item [(d)] A $k$-simple Bratteli diagram $B$ is non-elementary if and only if the infinite-path space does not have isolated points (so it is a Cantor set). To guarantee that a Bratteli diagram is non-elementary 
it suffices to require that for every infinite path $x = (x_i)$ there are infinitely
many edges $x_m$ such that $|s^{-1}(r(x_m))| >1$
\end{itemize}

(2) Let $B$ be a $k$-simple Bratteli diagram, then the sub-diagram restricted 
to the vertices in $V_o^n$, denoted by $I_B$, induces an ordered ideal ${\mathrm{K}}_{I_B}\subseteq {\mathrm{K}}_B$ 
such that
${\mathrm{K}}_B/{\mathrm{K}}_{I_B} \cong \bigoplus_{i=1}^k G_i$, where $G_i$ are simple dimension 
groups induced by the restriction of $B$ to the vertices in $V_i^n$.
Moreover, the diagram $B$ is strongly $k$-simple if and only if ${\mathrm{K}}_{I_B}$ is a 
simple dimension group; it is non-elementary if and only if ${\mathrm{K}}_{I_B}$ has no quotient which is 
isomorphic to $\Int$, and if and only if the the space of infinite paths through 
the sets $V^n_o$ is a locally compact Cantor set.
\end{rem}

An ordered Bratteli diagram $B=(V, E, >)$ is a Bratteli diagram with 
(partial) order $>$ on $E$ so that two edges $e$ and $e'$ are comparable 
if and only if $r(e)=r(e')$, see Definition \ref{order_definition}. Denote 
 by $E_{\textrm{max}}$ and 
$E_{\textrm{min}}$ the sets of maximal edges and minimal edges, 
respectively. This partial order induces a lexicographical partial order on 
paths (infinite or finite). Denote by $X_{\textrm{max}}$ and 
$X_{\textrm{min}}$ the set of maximal infinite paths and the set of minimal 
infinite paths respectively. Also note that if $B'$ is the Bratteli diagram 
obtained by telescoping on $B$, the lexicographic order on $B'$ make it into 
an ordered Bratteli diagram canonically.

\begin{defn}\label{defn-BD}
An ordered Bratteli diagram $B=(V, E, >)$ is called {\em $k$-simple} (with 
a slight abusing of notation) if it satisfies the following conditions:
\begin{enumerate}

\item\label{cond-unorder} the unordered Bratteli diagram $ B =(V, E)$ is 
$k$-simple in the sense of Definition \ref{defn-unordered-BD},

\item \label{minimalsets} There are infinite paths $z_{1, \textrm{max}}, ..., 
z_{k, \textrm{max}}$ and $z_{1, \textrm{min}}, ..., z_{k, \textrm{min}}$ 
such that for any level $n$ and $1\leq i\leq k$, $$\{z_{i, \textrm{min}}^n, 
z_{i, \textrm{max}}^n\}\subset V_{i}^n$$ and $$X_{\textrm{max}}=\{z_{1, 
\textrm{max}}, ..., z_{k, \textrm{max}}\},$$ $$X_{\textrm{min}}=\{z_{1, 
\textrm{min}}, ..., z_{k, \textrm{min}}\}.$$ 

By this condition and Lemma \ref{shorten} below, there is $L$ such that for 
all $n\geq L$ and any $v\in V_o^n$, the maximal edge (or minimal edge) 
starting with $v$ backwards to $V^1$ will end up in $V^1_i$ for some 
$1\leq i\leq k$. Denote by $m_+(v)=i$ (or $m_-(v)=i$). 

\item \label{cont-cond}
For any $v\in V^n_o$, one has
\begin{enumerate}
\item\label{prop-4-1} if $e$ is an edge with $s(e)=v$, one has 
$$m_-(s(e+1))=m_+(v),$$ (if $e\in E_{\max}$, the vertex $s(e+1)$ is 
understood as $s(e'+1)$ with $e'$ a non-maximal edge starting with $e$ and 
ending at some level $m>n$---such edge exists and $m_-(s(e+1))$ is well 
defined, by Condition \ref{minimalsets}), and
\item\label{prop-4-2} if $e$ is an edge with $e\notin E_{\max}$, $r(e)=v$ 
and $s(e)\in V^{n-1}_i$ with $n\geq 3$, one has $$m_-(s(e+1))=i.$$
\end{enumerate}
\end{enumerate}
If, in addition, the unordered Bratteli diagram $(V, E)$ is strongly 
$k$-simple, then $B$ is said to be a strongly $k$-simple ordered Bratteli
 diagram.
\end{defn}

\begin{rem}
Note that if $k=1$, then Condition \ref{cont-cond} is redundant. Moreover, 
Condition \ref{cont-cond} is preserved under telescoping. 
\end{rem}

\begin{lem}\label{shorten}
Any ordered Bratteli diagram satisfying Condition \ref{minimalsets} 
of Definition \ref{defn-BD} can be 
telescoped to an ordered Bratteli diagram satisfying the following condition: if 
$e$ and $e'$ are in $E_{\textrm{max}}$ (or $E_{\textrm{min}}$) with 
$r(e)=s(e')$, then $e$ is in $X_{\textrm{max}}$ (or $X_{\textrm{min}}$). 
\end{lem}

\begin{proof}
The proof is similar to that of Proposition 2.8 of \cite{HPS-Cantor}. Let $T$ 
denote the graph obtained from $E_{\textrm{max}}$ by deleting $z_{1, 
\textrm{max}} ,..., z_{k, \textrm{max}}$. By Condition \eqref{minimalsets}, 
each connected component of $T$ is finite.

Let $n_0=0$. Having defined $n_k$, choose $n_{k+1}$ so that no vertex in 
$V_{n_k}$ is connected to all vertices of $V_{n_{k+1}}$. Contract 
the diagram to the subsequence $\{n_k;\ k\geq 0\}$. Then, this diagram 
satisfies the lemma.
\end{proof}

Recall that two ordered Bratelli diagrams $B_1$ and $B_2$ are 
\textit{equivalent} if 
there is an ordered Bratteli diagram $B$ such that $B_1$ and $B_2$ can be 
obtained by telescoping on $B$. 
\begin{lem}\label{OBD-red}
Let $B=(V, E, >)$ be a $k$-simple ordered Bratteli diagram. Then it is 
equivalent to a $k$-simple Bratteli diagram $B'=(V', E', >')$ satisfying the 
following conditions:
\begin{enumerate}
\item\label{tel-cond-1}  if $e\in E'_{\max}$ with $r(e)\in V_o'$, then $s(e)
\notin V_o'$ (so the edge $e$ in Condition \ref{prop-4-1} cannot be a 
maximal edge in $B'$), 
\item\label{tel-cond-2} for any $1\leq i\leq k$ and any $n\geq 1$, each 
vertex $v\in (V')_i^{n+1}$ is connected to all vertices of 
$w\in (V')_{i}^{n}$, and
\item\label{tel-cond-3} if $B$ is strongly $k$-simple, then, for any $n\geq 
1$, each vertex of $v\in (V')^{n+1}_o$ is connected to all vertices of $
(V')_o^n$ (and hence to all vertices of $(V')^n$).
\end{enumerate}
Moreover, if $B$ is an unordered $k$-simple Bratteli diagram, then it is 
equivalent to an unordered $k$-simple Bratteli diagram $B'$ which satisfies 
Condition \ref{tel-cond-2} and Condition \ref{tel-cond-3}.
\end{lem}

\begin{proof}
Condition \ref{tel-cond-1} follows from Lemma \ref{shorten}. Since $B$ is 
$k$-simple, by \ref{cond2} of Definition \ref{defn-unordered-BD}, Condition 
\ref{tel-cond-2} can also be obtained by a telescoping of $B$. 

If $B$ is strongly $k$-simple. Then $B$ can be telescoped further so that if a 
vertex $v\in V^{n+1}_o$ is connected to a vertex of $V^n_o$, then it is 
connected to all vertices of $V^n_o$. For Condition \ref{tel-cond-3}. One needs to  find an equivalent Bratteli diagram $B'$ so that each vertex $v\in (V')^{n+1}_o$ is 
connected to all vertices of $(V')^n_o$.

For each $n\geq 1$, write
$$V_o^{n+1}=\{(v')^{n+1}_1, ..., (v')^{n+1}_{r_{n+1}}, v^{n+1}_1, ..., 
v^{n+1}_{t_{n+1}}\}$$ where the vertices $v'_i$ are not connected 
to $V_o^n$. Denote by $w^{n}_1, ..., w^{n}_l$ the vertices 
from $V^n$ which are  not in $V^n_o$.  
Interpolate the level $n$ and $n+1$ of $B$ as follows: Consider the 
vertices
$$\tilde{V}^n:\ w_1^n, ..., w_l^n,  v^{n+1}_1, ..., v^{n+1}_{t_{n+1}}.$$
The map from $V^n\to\tilde{V}^n$ is defined as the identity if 
restricted to $w_1^n, ..., w_l^n$ (in $V^n$), and the original map 
(that is, it is defined by edges of $B$) if 
restricted to $v^{n+1}_1, ..., v^{n+1}_{t_{n+1}}$. Define the map from 
$\tilde{V}_n\to V^{n+1}$ as the original map if restricted to $(v')^{n+1}_1, 
..., (v')^{n+1}_{r_{n+1}}$, and the identity if restricted to 
$v^{n+1}_1, ..., v^{n
+1}_{t_{n+1}}$ (in $V_{n+1}$). It can be illustrated in the following picture. 
The original maps
$$
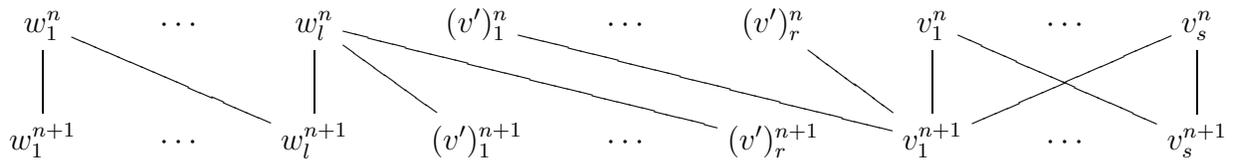
\xymatrix{
 w^n_1 \ar@{-}[d] \ar@{-}[drr] &  \cdots & w^n_l \ar@{-}[d] \ar@{-}[dr] 
 \ar@{-}[drrr] & (v')^n_1  \ar@{-}[drrr] & \cdots & (v')_r^n \ar@{-}[dr] & 
 v^{n}_1  \ar@{-}[drr] \ar@{-}[d] & \cdots & v^{n}_s \ar@{-}[d] \ar@{-}[dll] 
 \\
 w^{n+1}_1 &  \cdots & w^{n+1}_l & (v')^{n+1}_1 & \cdots & (v')_r^{n+1} 
 & v^{n+1}_1 & \cdots & v^{n+1}_s 
}
$$
are interpolated into
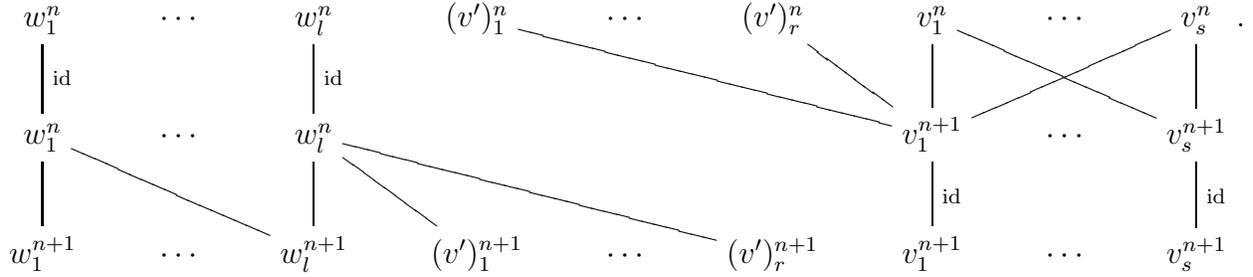
$$
\xymatrix{
 w^n_1 \ar@{-}[d]^{\mathrm{id}} &  \cdots & w^n_l \ar@{-}
 [d]^{\mathrm{id}}  & (v')^n_1 \ar@{-}[drrr]  & \cdots & (v')_r^n \ar@{-}
 [dr] & v^{n}_1 \ar@{-}[drr] \ar@{-}[d] & \cdots & v^{n}_s \ar@{-}[d] 
 \ar@{-}[dll] \\
 w^n_1 \ar@{-}[d] \ar@{-}[drr] &  \cdots & w^n_l \ar@{-}[d] \ar@{-}[dr] 
 \ar@{-}[drrr] &  &  &  & v^{n+1}_1 \ar@{-}[d]^{\mathrm{id}} & \cdots & 
 v^{n+1}_s \ar@{-}[d]^{\mathrm{id}} \\
 w^{n+1}_1 &  \cdots & w^{n+1}_l & (v')^{n+1}_1 & \cdots & (v')_r^{n+1} 
 & v^{n+1}_1 & \cdots & v^{n+1}_s 
}.
$$


Put an order on this enlarged Bratelli diagram as the following: Consider a vertex of the level $\tilde{V}^n$. If it is one of the $w_i^n$, $i=1, ..., l$, then there is only one edge connecting it backwards, and so just put the trivial order; if it is one of the $v^{n+1}_i$, $s=1, ..., s$,  then the edges backwards are exactly the same edges backwards as in the original Bratelli diagram, and so just put the identical order as it is in the original Bratteli diagram. Put the order in the similar way for vertices of the level $V^n$. Then it is straightforward to check that  its telescoping  into the levels $V^{n}$ is exactly the original ordered Bratteli diagram $B$.

Denote by $B'$ the telescoping of this diagram into the levels $\tilde{V}^n$. Note that the vertices of $(V')_0^n$ only consist of $v^{n+1}_1, ..., v^{n+1}_{t_{n+1}}$, and hence $B'$ is a
 $k$-simple ordered Bratteli diagram satisfying Condition \ref{tel-cond-3}. It follows from Lemma \ref{shorten} that Condition \ref{tel-cond-1} 
 and Condition \ref{tel-cond-2} can also be satisfied by a further telescoping.
\end{proof}

\begin{rem}
In the rest of the paper, we always assume that strong $k$-simple Bratteli 
diagrams satisfy Conditions \ref{tel-cond-1}, \ref{tel-cond-2}, and \ref{tel-cond-3} of Lemma \ref{OBD-red}. 
\end{rem}

\begin{example}\label{ex-2-BD}
The following is an example of a strongly $2$-simple ordered Bratteli
 diagram (with level $0$ omitted):

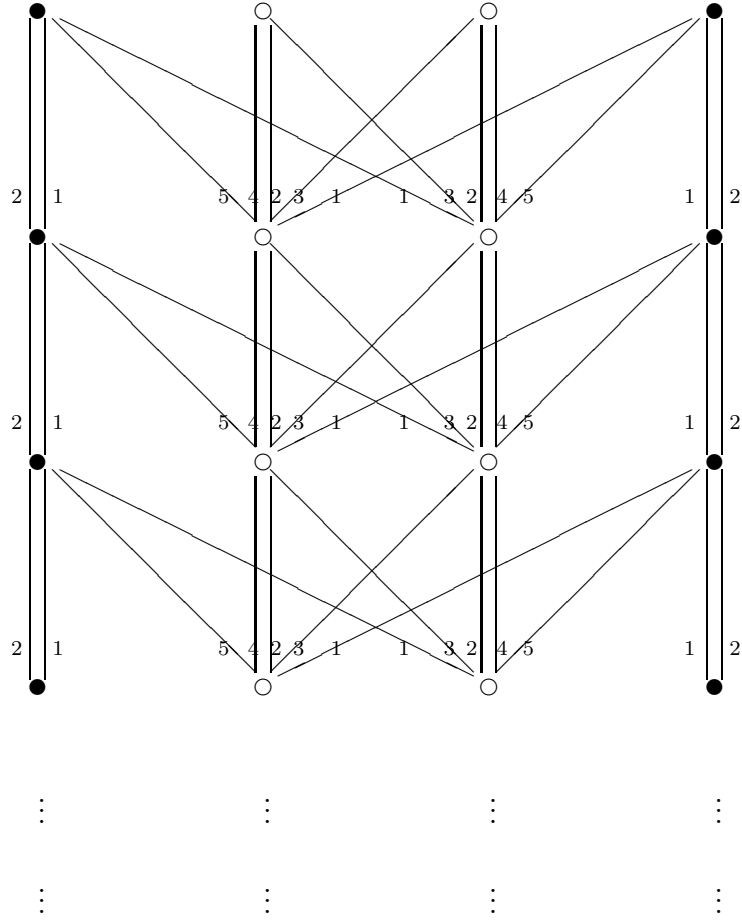
\begin{center}
\setlength{\unitlength}{1mm}
\begin{picture}(90, 120)

\put(0,0){$\vdots$}
\put(30,0){$\vdots$}
\put(60,0){$\vdots$}
\put(90,0){$\vdots$}

\put(0,12){$\vdots$}
\put(30,12){$\vdots$}
\put(60,12){$\vdots$}
\put(90,12){$\vdots$}

\put(0,30){\circle*{2}}
\put(30,30){\circle{2}}
\put(60,30){\circle{2}}
\put(90,30){\circle*{2}}

\put(0,60){\circle*{2}}
\put(30,60){\circle{2}}
\put(60,60){\circle{2}}
\put(90,60){\circle*{2}}

\put(0,90){\circle*{2}}
\put(30,90){\circle{2}}
\put(60,90){\circle{2}}
\put(90,90){\circle*{2}}

\put(0,120){\circle*{2}}
\put(30,120){\circle{2}}
\put(60,120){\circle{2}}
\put(90,120){\circle*{2}}

\put(-1,31){\line(0,1){28}}
\put(1,31){\line(0,1){28}}

\put(-1,61){\line(0,1){28}}
\put(1,61){\line(0,1){28}}

\put(-1,91){\line(0,1){28}}
\put(1,91){\line(0,1){28}}

\put(89,31){\line(0,1){28}}
\put(91,31){\line(0,1){28}}

\put(89,61){\line(0,1){28}}
\put(91,61){\line(0,1){28}}

\put(89,91){\line(0,1){28}}
\put(91,91){\line(0,1){28}}

\put(29,32){\line(0,1){26}}
\put(31,32){\line(0,1){26}}

\put(29,62){\line(0,1){26}}
\put(31,62){\line(0,1){26}}

\put(29,92){\line(0,1){26}}
\put(31,92){\line(0,1){26}}

\put(59,32){\line(0,1){26}}
\put(61,32){\line(0,1){26}}

\put(59,62){\line(0,1){26}}
\put(61,62){\line(0,1){26}}

\put(59,92){\line(0,1){26}}
\put(61,92){\line(0,1){26}}

\put(3, 119){\line(2, -1){55}}
\put(2, 119){\line(1, -1){27}}

\put(3, 89){\line(2, -1){55}}
\put(2, 89){\line(1, -1){27}}

\put(3, 59){\line(2, -1){55}}
\put(2, 59){\line(1, -1){27}}

\put(87, 119){\line(-2, -1){55}}
\put(88, 119){\line(-1, -1){27}}

\put(87, 89){\line(-2, -1){55}}
\put(88, 89){\line(-1, -1){27}}

\put(87, 59){\line(-2, -1){55}}
\put(88, 59){\line(-1, -1){27}}

\put(31,62){\line(1, 1){27}}
\put(31, 89){\line(1, -1){27}}

\put(31,92){\line(1, 1){27}}
\put(31, 119){\line(1, -1){27}}

\put(31,32){\line(1, 1){27}}
\put(31, 59){\line(1, -1){27}}

\put(24, 95){${}_5$}
\put(28, 95){${}_4$}
\put(39, 95){${}_1$}
\put(31, 95){${}_2$}
\put(34, 95){${}_3$}

\put(-3.5, 95){${}_2$}
\put(2, 95){${}_1$}
\put(86, 95){${}_1$}
\put(92, 95){${}_2$}

\put(48, 95){${}_1$}
\put(54, 95){${}_3$}
\put(57, 95){${}_2$}
\put(61, 95){${}_4$}
\put(64.5, 95){${}_5$}

\put(-3.5, 65){${}_2$}
\put(2, 65){${}_1$}
\put(86, 65){${}_1$}
\put(92, 65){${}_2$}

\put(24, 65){${}_5$}
\put(28, 65){${}_4$}
\put(39, 65){${}_1$}
\put(31, 65){${}_2$}
\put(34, 65){${}_3$}

\put(48, 65){${}_1$}
\put(54, 65){${}_3$}
\put(57, 65){${}_2$}
\put(61, 65){${}_4$}
\put(64.5, 65){${}_5$}

\put(-3.5, 35){${}_2$}
\put(2, 35){${}_1$}
\put(86, 35){${}_1$}
\put(92, 35){${}_2$}

\put(24, 35){${}_5$}
\put(28, 35){${}_4$}
\put(39, 35){${}_1$}
\put(31, 35){${}_2$}
\put(34, 35){${}_3$}

\put(48, 35){${}_1$}
\put(54, 35){${}_3$}
\put(57, 35){${}_2$}
\put(61, 35){${}_4$}
\put(64.5, 35){${}_5$}

\end{picture}
\end{center}

\end{example}

\subsection{Bratteli-Vershik map}
Let $B$ be a non-elementary ordered $k$-simple Bratteli diagram, 
and denote by 
$X_B$ the space of infinite paths of $B$. For each finite path $\xi$, we
will denote by $\chi_\xi$ the cylinder set consisting of all paths starting 
with $\xi$. Then $X_B$ is a Cantor set with topology generated by these 
cylinder sets because the cylinder sets are clopen $X_B$. 

Let us adapt the well-known construction of the Vershik map $\sigma: X_B
\to X_B$ of the simple case to the 
case of $k$-simple Bratteli diagrams.
Let $\xi=(\xi^1, \xi^2, ...)\in X_B$ with each $\xi^i\in E$. If $\xi=z_{i, 
\textrm{max}}$ for some $1\leq i\leq k$, then define 
$$\sigma(\xi)=\sigma(z_{i, \textrm{max}})=z_{i, \textrm{min}}.$$
Otherwise, set $$d(\xi)=\max\{m;\ (\xi^1, ..., \xi^{m-1})\in E_{\max}\},$$ 
and for any vertex $v$ at level $n$, set $r_{\min}(v)\in E_{\min}^{1, n}$ the minimal edge with range $v$. Then define
$$\sigma(\xi)(n)=\left\{
\begin{array}{ll}
r_{\min}(s(\xi^{d(\xi)}+1)),& \textrm{if}\ n< d(\xi);\\
\xi^n+1, & \textrm{if}\ n= d(\xi);\\
\xi^n, & \textrm{if}\ n\geq d(\xi)+1.
\end{array}
\right.
$$

\begin{lem}
The map $\sigma: X_B \to X_B$ is a homeomorphism.
\end{lem}
\begin{proof}
Define a map $\tau: X_B\to X_B$ as follows: Let $\xi=(\xi^1, \xi^2, ...)\in X_B$ with each $\xi^i\in E$. If $\xi=z_{i, \textrm{min}}$ for some $1\leq i\leq k$, then define 
$$\sigma(\xi)=\sigma(z_{i, \textrm{min}})=z_{i, \textrm{max}}.$$
Otherwise, set $$c(\xi)=\max\{m;\ (\xi^1, ..., \xi^{m-1})\in E_{\min}\},$$ and for any vertex $v$ at level $n$, set $r_{\max}(v)\in E_{\max}^{1, n}$ the maximal edge with range $v$. Then define
$$\tau(\xi)(n)=\left\{
\begin{array}{ll}
r_{\max}(s(\xi^{c(\xi)}+1)),& \textrm{if}\ n< d(\xi);\\
\xi^n-1, & \textrm{if}\ n= d(\xi);\\
\xi^n, & \textrm{if}\ n\geq d(\xi)+1.
\end{array}
\right.
$$
Then it is straightforward to calculation that $\sigma\circ\tau=\tau\circ\sigma=\mathrm{id}$, and thus the map $\sigma$ is one-to-one and onto.

Since $X_B$ is a compact metrizable space, to show that $\sigma$ is a homeomorphism, it is enough to show that $\sigma$ is continuous. It is clear that $\sigma$ is continuous at each $\xi\in X_{B}\setminus X_{\max}$. Consider $z_{i, \max}$ and a sequence $\xi_j \to z_{i, \max}$ with $\xi_j\notin X_{\max}$. Let $N\in\mathbb N$. Then there is $J$ so that for any $j>J$, one has $$\xi_j(n)=z_{i, \max}(n),\quad\forall 1\leq n\leq N.$$

Pick an arbitrary $\xi_j$ with $j>N$, and put $$M=\max\{n; \xi(n')\in E_{\max}\ \forall n' \leq n\}.$$ Note that $M\geq N$. If $r(\xi_j(M+1)) \in V^{M+2}_i$, then one has $$\sigma(\xi_j)(n)=z_{i, \min}(n)=\sigma(z_{i, \max})(n),\quad\forall 1\leq n\leq M-2.$$

If $r(\xi_j(M+1))\in V^{M+2}\setminus \bigcup_i V_i^{M+2}$ and $\xi_j(M+1)\in E_{\max}$. By Condition \ref{prop-4-1}, one has that $$\sigma(\xi_j)(n)=z_{i, \min}(n)=\sigma(z_{i, \max})(n),\quad\forall 1\leq n\leq M-1.$$ 

If $r(\xi_j(M+1))\in V^{M+2}\setminus \bigcup_i V_i^{M+2}$ and $\xi_j(M+1)\notin E_{\max}$. By Condition \ref{prop-4-2}, one still has that $$\sigma(\xi_j)(n)=z_{i, \min}(n)=\sigma(z_{i, \max})(n),\quad\forall 1\leq n\leq M-1.$$

Since $M\geq N$, one has that  $\sigma(\xi_j)$ is in the $N-1$ neighbourhood of $z_{i, \min}$ for any $j\geq J$, and hence the map $\sigma$ is continuous at $z_{i, \max}$, as desired.
\end{proof}

\begin{thm}
The Bratteli-Vershik system $(X_B, \sigma)$ has $k$ minimal subsets.
\end{thm}
\begin{proof}
For each $1\leq i\leq k$, denote by $Y_i$ the closed subset corresponding to the paths $\{z;\ z^n\in V^n_i\}$. It follows from Condition \ref{cond1} that $Y_1, ..., Y_k$ are closed invariant subsets. By Condition \ref{cond2}, the sets $Y_1, ..., Y_k$ are minimal. Note that $z_{i, \max}\in Y_i$ for each $i$, and the orbit of $z_{i, \max}$ is dense in $Y_i$

Let $U$ be any minimal invariant closed nonempty subset of $X_B$. 

Pick any point $x=(x_1, x_2, ...)\in U$, and fix $n\in\mathbb N$ such that the $n$-neighbourhoods of each $z_{i, \textrm{max}}$ are pairwise disjoint. The finite path $(x_1,...,x_{n+1})$ has finitely many successors in $P_{0, n+1}$. In particular, its $i\mbox{th}$ successor is maximal in $P_{0, n+1}$ for some $i$, and therefore $\sigma^i(x)=(f_1, ..., f_{n+1})$ with $(f_1, ..., f_{n+1})$ maximal. Since the diagram satisfies Lemma \ref{shorten}, the finite path $(f_1, ..., f_{n})$ is in $X_{\mathrm{max}}$, and hence in the $n$-neighbourhood of $z_{i, \mathrm{max}}$ for some $i$.

Therefore, at least one of $\{z_{1, \textrm{max}}, ..., z_{k, \textrm{max}}\}$ is in the closure of the orbit of $x$, and hence $U$ has to contain one of $\{Y_1,...,Y_k\}$. Since $U$ is also minimal, it has to be one of $\{Y_1,...,Y_k\}$, as desired.
\end{proof}

Consider the C*-algebra $A_{y_1, ..., y_k}$, and write $V^n=\{v_1, ..., v_{\abs{V^n}}\}$. For each $v_i$, denote by $f_1^i<\cdots<f^i_{l_i}$ the finite paths ending at $v_i$ (they form a totally ordered set). Consider
$$F_n:=\bigoplus_{i=1}^{\abs{V^n}}\textrm{C*}\{\chi_{f^i_1}, \chi_{f^i_1}u, ..., \chi_{f^i_{l}}u^{l_i-1}\}\subseteq A_{y_1, ..., y_k}.$$

\begin{lem}\label{approx-l-subalg}
The sub-C*-algebras $\{F_n\}$ have the following properties:
\begin{enumerate}
\item\label{subalg-propty-00} $F_n \cong \bigoplus_{i=1}^{\abs{V^n}} \mathrm{M}_{l_i}(\Comp)$,
\item\label{subalg-propty-01} $F_1\subseteq\cdots\subseteq F_n\subseteq F_{n+1}\subseteq\cdots,$
\item\label{subalg-propty-02} $\bigcup_n F_n$ is dense in $A_{y_1, ..., y_k}$,
\item\label{subalg-propty-03} the $\Kzero$-map induced by the inclusion $F_n\subseteq F_{n+1}$ is the same as the multiplicities between $V^n$ and $V^{n+1}$.
\end{enumerate}
\end{lem}

\begin{proof}
Property \eqref{subalg-propty-00} is clear and Property \eqref{subalg-propty-02} is standard (for instance, see \cite{Poon-Rocky}). Let us only show Property \eqref{subalg-propty-01} and Property \eqref{subalg-propty-03}.

Let $f$ be a minimal finite path with end point $v$. Assume that $v$ is sent to $v_1, ..., v_l$ with multiplicity $m_1, ..., m_l$ at level $n+1$. Consider $\chi_fu^l$ with $l$ strict smaller than the number of paths ending at $v$, and denote the edges between $v$ and $v_i$ by $f_1^{i}< \cdots < f_{m_i}^i$. Then
\begin{eqnarray*}
\chi_fu^l &=& (\sum_{i=1}^l\sum_{j=1}^{m_i} \chi_{ff^i_j})u^l\\
&=& \sum_{i=1}^l\sum_{j=1}^{m_i} (u^*)^{j-1}\chi_{ff^i_1}u^{j-1}u^l\\
&=&\sum_{i=1}^l\sum_{j=1}^{m_i} ((u^*)^{j-1}\chi_{ff^i_1})(\chi_{ff^i_1}u^{l+j-1})\in F_{n+1}.
\end{eqnarray*}
Thus, $F_n\subseteq F_{n+1}$. 

Let us calculate the $\Kzero$-map. Applying the equation above for $l=0$, one has
$$[\chi_f]= \sum_{i=1}^l\sum_{j=1}^{m_i} [(u^*)^{j-1}\chi_{ff^i_1}\chi_{ff^i_1}u^{j-1}]= \sum_{i=1}^l{m_i} [\chi_{ff^i_1}].$$

Since the standard generators of $\Kzero(F_n)$ are $[\chi_f]$, the $\Kzero$-map agrees with the multiplicity map between $V^{n}$ and $V^{n+1}$, as desired.
\end{proof}

As a straightforward corollary, one has
\begin{thm}\label{K0-model}
Denote by $\mathrm{K}_B$ the dimension group associated with $B$.
One then has that $$\Kzero(A_{y_1,...,y_k})\cong\mathrm{K}_B$$ as ordered groups.
\end{thm}

\section{From Cantor systems to Bratteli-Vershik models}\label{Can-BD}

%
%
%

In this section, we shall show that any Cantor system 
$(X, \sigma)$ with finitely many 
minimal compenents can be modeled by the Bratteli-Vershik map on an ordered 
$k$-simple Bratteli diagram which is introduced in the previous section (see Theorem \ref{model-k-simple}, Theorem \ref{model} below). 
Results of this kind, based on sequences of Kakutani-Rokhlin partitions, 
have been discussed in a number of papers; see, for example, 
\cite{HPS-Cantor, Medynets2006, BezuglyiKwiatkowskiMedynets2009}. 
Nevertheless, we find it useful for the reader to see a complete proof where
all details are clarified.

We first show that if, for a given Cantor system $(X, \sigma)$,
 a sequence of Kakutani-Rokhlin partitions 
satisfies certain conditions, then it determines an ordered Bratteli diagram 
$B$. Then we prove that if $(X,\sigma)$ has $k$ minimal subsets, then 
$B$ is a $k$-simple Bratteli diagram, described in Section \ref{BV-model}.

\begin{defn}
A \textit{Kakutani-Rokhlin partition} of $(X, \sigma)$ consists of  pairwise 
disjoint clopen $\sigma$-towers 
 $$\xi_l := \{Z(l, j)\ ;\  1\leq j\leq J(l)\},\quad 1\leq l\leq L,
 $$ 
of height  $J(l)$  such that 
\begin{enumerate}
\item $Z(l, j) \cap Z(l, j') = \O,$ $j \neq j'$,
\item\label{kr-cond1} $\bigcup_{l, j} Z_{l, j}=X$, and
\item\label{kr-cond2} $\sigma(Z(l, j))=Z(l, j+1)$ for any $1\leq j< J(l)$.
\end{enumerate}
\end{defn}

\begin{rem}
Denote by $Z=\bigcup_{l=1}^L Z(l, J(l))$. Then one has $$
\bigcup_{l=1}^L Z(l, 1)=\sigma(Z).$$
\end{rem}

\begin{lem}[Lemma 4.1 of \cite{HPS-Cantor}]\label{pre-partition}
Let $Z$ be a clopen subset such that $y_i\in Z$ for any $1\leq i\leq k$, and 
let $\mathscr P$ be a partition of $X$ into  clopen sets. Then, there is a 
Kakutani-Rokhlin partition $$\{Z(l, j);\ 1\leq l\leq L, 1\leq j\leq J(l)\}$$ of 
$X$ which is finer than $\mathscr P$ and $Z=\bigcup_{l=1}^L Z(l, J(l))$.
\end{lem}
\begin{proof}
By Lemma \ref{AF-cond}, one has that $X=\bigcup_{i=-N}^N\sigma^i(Z)$ for 
some $N$. Applying $\sigma^{-N}$ on both sides, one has that
\begin{equation}\label{rec-00}
X=\bigcup_{i=0}^{2N}\sigma^{-i}(Z).
\end{equation}
For each $x\in Z$, define $$r(x)=\min\{i\geq 1;\ \sigma^i(x)\in Z\}.$$ By 
\eqref{rec-00}, the map $r: Z\to \mathbb N$ is well defined. Moreover, the 
map $r$ is continuous. Write $$r(Z)=\{J(1), J(2), ... , J(L)\},$$ and define 
$$Z(l, j)=\sigma^{j}(r^{-1}(J(l))).$$

It is clear that the sets $\{Z(l, j)\}$ are clopen, and satisfy Condition \ref{kr-cond2}. Show that they form a partition of $X$. If $x\in Z(l_1, j_1)
\cap Z(l_2, j_2)$, then there are $y_1\in r^{-1}(J(l_1))$ and $y_2\in r^{-1}
(J(l_2))$ such that $$\sigma^{j_1}(y_1)=\sigma^{j_2}(y_2).$$
If $j_1>j_2$, then one has $$\sigma^{j_1-j_2}(y_1)=y_2\in r^{-1}(J(l_2))
\subseteq Z,$$ which contradicts to $y_1\in r^{-1}(J(l_1))$. If $j_1<j_2$, 
the same argument leads to a contradiction. If $j_1=j_2$, then one has that 
$y_1=y_2$ and hence $l_1=l_2$. Thus , the collection of sets $\{Z(l, j)\}$ 
consists of pairwise disjoint.

For any $x\in X\setminus Z$, consider $$n=\min\{i\geq0; \sigma^{-i}(x)\in 
Z\}\quad\textrm{and}\quad m=\min\{i\geq0; \sigma^{i}(x)\in Z\}.$$ By 
Lemma \ref{AF-cond}, such $n$ and $m$ exist. Note that $m, n\geq 1$ and 
$$\sigma^{-n}(x)\in r^{-1}(m+n+1).$$ Therefore,  $x\in Z(l, n)$ with 
$J(l)=m+n+1.$ 

For any $x\in Z$, consider $$n=\min\{i\geq 1;\ \sigma^{-i}(x)\in Z\}.$$
 Then one has that $r(\sigma^{-n}(x))=n$ and 
\begin{equation}\label{part-inc-00}
x\in Z(l, J(l))
\end{equation} 
for $l$ with $J(l)=n$.
Hence $\{Z(l, j)\}$ s a partition of $X$, and they form a Kakutani-Rokhlin 
partition of $X$ with respect to $\sigma$.

It is clear that $\bigcup_l Z(l, J(l))\subseteq Z$ for any $1\leq l\leq L$ by 
the construction. On the other hand, it follows from \eqref{part-inc-00} that 
$\bigcup_l Z(l, J(l))\supseteq Z$ and hence one has 
$$\bigcup_l Z(l, J(l))= Z.$$

Once we have a Kakutani-Rokhlin partition of $X$ with respect to $\sigma$, a 
similar argument as that  of \cite[Lemma 3.1]{Put-PJM} shows that $\{Z(l, j)
\}$ always can be modified further so that $\{Z(l, j)\}$ is finer than the given 
partition $\mathscr P$. 
\end{proof}

\begin{thm}[Theorem 4.2 of \cite{HPS-Cantor}]
\label{RH-pa}
There are Kakutani-Rokhlin partitions of $X$ $$\mathscr P_n=\{Z(n, l, j);\ 
1\leq l\leq L(n), 1\leq j\leq J(n, l)\}$$ such that
\begin{enumerate}
\item\label{decrease-seq} the sequence $(Z_n:=\bigcup_{l=1}^{L(n)} Z(n, l, 
J(n, l)))$ is a decreasing sequence of clopen sets with intersection $\{y_1, 
y_2, ..., y_k\}$, where the points $\{y_1, y_2, ..., y_k\}$ are chosen
in minimal components $Y_1, ... , Y_k$ respectively,
\item the partition $\mathscr P_{n+1}$ is finer than the partition $\mathscr 
P_n$, and
\item $\bigcup_n\mathscr P_n$ generates the topology of $X$.
\end{enumerate} 
\end{thm}

\begin{proof}
Choose a sequence of clopen sets $Z_1\supseteq Z_2\supseteq\cdots$ with 
$\bigcap_n Z_n=\{y_1, y_2, ..., y_k\}$, and a sequence of finite partitions $
(\mathscr P'_n)$ such that $\bigcup\mathscr P'_n$ generates the topology 
of $X$. By Lemma \ref{pre-partition}, there is a Kakutani-Rokhlin partition $
\mathscr P_1:=\{Z(1, l, j); 1\leq l\leq L(1), 1\leq j\leq J(1, l)\}$ such that $
\bigcup_l Z(1, l, J(1, l))=Z_1$ and $\mathscr P_1$ is finer than $\mathscr 
P_1'$.

Assume that the Kakutani-Rokhlin partition $\mathscr P_1, ..., \mathscr 
P_{n-1}$ are constructed. Then, by Lemma \ref{pre-partition}, there is a 
Kakutani-Rokhlin partition $\mathscr P_n:=\{Z(n, l, j); 1\leq l\leq L(n), 1\leq j
\leq J(n, l)\}$ such that $\bigcup_l Z(n, l, J(n, l))=Z_n$ and $\mathscr P_n$ 
is finer than $\mathscr P_n'\vee\mathscr P_{n-1}$. Then $(\mathscr P_n)$ 
is the desired sequence of Kakutani-Rokhlin partition.
\end{proof}

Based on the sequence of Kakutani-Rokhlin partitions, one is able to construct an 
ordered Bratteli diagram $B= (V, E, >)$  following the procedure 
 described in 
Section 4 of \cite{HPS-Cantor}:

 For convenience, we may assume that $L(0)= 1$, 
$J(0, 1)=1$, and $Z(0, 1, 1)= X$. The set of vertices $V^n$ of the 
diagram $B$ is  formed by the towers in the 
Kakutani-Rokhlin partition $\mathscr P_n$. To define the set of edges, we 
say  that there is an edge between a
 tower (vertex) $\xi_i^{n-1}$ at the level $n-1$ and a tower (vertex) 
 $\xi_j^n$ at the level $n$ if 
 $\xi_j^n$ pass through $\xi_i^{n-1}$.  More precisely, for each $n$, one
  has 
$$V^n=\{(n, 1), (n, 2), ..., (n, L(n))\},$$ and
$$E^n=\{(n, l, l', j')| Z(n, l', j'+j)\subseteq Z(n-1, l, j)\ \textrm{for all 
$j=1, ..., J(n-1, l)$}\}.$$
Then the source and range maps are
$$s((n, l, l', j'))=(n-1, l)$$ and
$$r((n, l, l', j'))=(n, l').$$
The order on the edges comes from the natural order on each tower 
$\{Z(n, l, j);\ j=1, ..., J(n, l)\}$. That is,
$$(n, l_1, l', j'_1)>(n, l_2, l', j'_2)\ \textrm{if and only if}\ j'_1> j'_2.
$$ 

\begin{thm}\label{model-k-simple} 
Let $(X, \sigma)$ be a 
Cantor dynamical system with $k$ minimal sets $Y_1, ..., Y_k$.
The ordered Bratteli diagram $B =(V, E, >)$, constructed as above, is non-elementary, and 
 satisfies Condition \ref{cond-unorder}--\ref{cont-cond} of Definition 
 \ref{defn-BD}, that is, $B$ is a non-elementary $k$-simple ordered Bratteli diagram.
\end{thm}

\begin{proof}
For each $n$ and $1\leq i\leq k$, set $$V_i^n:=\{(n, l);\ Y_i\cap Z(n, l, j))
\neq\O\ \textrm{for some $1\leq j\leq J(n, l)$}\}.$$ Note that since $Y_i$ 
is invariant, $$V_i^n =\{(n, l);\ Y_i\cap Z(n, l, j))\neq\O\ \textrm{for all 
$1\leq j\leq J(n, l)$}\}.$$ Since $Y_i$ are pairwise disjoint, by choosing $n$ 
sufficiently large, one has that $\{V^n_1, V^n_2, ..., V^n_k\}$ are pairwise 
disjoint.

Consider an edge $(n, l, l', j')$ with its range $(n, l')\in V_i^n$. One then has 
that $Z(n, l', j)\cap Y_i\neq\O$ for all $1\leq j\leq J(n, l')$, and thus 
$$Z(n-1, l, j)\cap Y_i\neq\O$$ for all $1\leq j \leq J(n-1, l)$. That is, the 
vertex $(n-1, l)$ is in $V_i^{n-1}.$ Hence the Bratelli diagram satisfies 
Condition \ref{cond1} of Definition \ref{defn-unordered-BD}.

Note that the restriction of $\sigma$ to $Y_i$ is minimal. Thus, Condition 
\ref{cond2} of Definition \ref{defn-unordered-BD} also holds. That is, the 
unordered diagram $(V, E)$ is $k$-simple, and this verified Condition 
\ref{cond-unorder} of \ref{defn-BD}.

Let $((n, l_n, l_n', j_n'))$ be an infinite path in $E_{\mathrm{min}}$. Since it 
is an infinite path, one has that $l'_{n-1}=l_n$. Since $(n, l_n, l_n', j_n')$ is 
minimal, one has that $j_n'=0$. Hence $Z(n, l_n, 1)\subseteq Z(n-1, l_{n-1}, 
1)$ for any $n$, and thus
$$\bigcap_{n=1}^\infty Z(n, l_n, 1)\subseteq\bigcap_{n=1}^
\infty(\bigcup_{l=1}^{L(n)} Z(n, l, 1))=\{\sigma(y_1), \sigma(y_2), ..., 
\sigma(y_k)\}.$$

By the construction of $Z(n, l, j)$, it is finer than any given partition $
\mathscr P'$. Therefore the intersection $\bigcap_{n=1}^\infty Z(n, l_n, 1)$ 
is a single point, and there is $1\leq i\leq k$ such that $$\bigcap_{n=1}^
\infty Z(n, l_n, 1)=\sigma(y_i).$$ This uniquely determines $l_n$, and thus 
there are at most $k$ infinite minimal paths. On the other hand, for each $y_i
$, the infinite sequence $((Z(n, l_n, 1)))$ with $\sigma(y_i)\in Z(n, l_n, 1)$ 
clearly form a minimal path. Thus, there are $k$ minimal paths.

A similar argument show that there are also $k$ maximum paths. Therefore, 
the ordered Bratteli diagram $(V, E, >)$ satisfies Condition 
\ref{minimalsets} of Definition \ref{defn-BD}.

For each $1\leq i\leq k$ and each $n\geq 1$, define $$N_{n, i}=\bigcup_{(n, 
l)\in V_i^n} \bigcup_{j=1}^{J(n, l)} Z(n, l, j).$$ Then each $N_{n, i}$ is a 
clopen neighbourhood of $Y_i$, and $N_{n, 1}, ..., N_{n, k}$ are pairwise 
disjoint.

Consider a vertex (a tower) $(n, l)\in V_n\setminus(V_1^n\cup\cdots\cup 
V_k^n)$. Note that if $n\geq 2$, by Condition \ref{decrease-seq} of 
Theorem \ref{RH-pa}, there is one and only one $1\leq i\leq k$ and one and 
only one $1\leq j\leq k$ such that 
$$Z(n, l, 1)\subseteq N_{n-1, i}\quad\textrm{and}\quad Z(n, l, J(n, l))
\subseteq N_{n-1, j}.$$ Then, by the definition, one has
$$m_-((n, l))=i\quad\textrm{and}\quad m_+((n, l))=j.$$

Let $(n+1, l, l', j')$ be an edge such that $$Z(n+1, l', j'+j)\subseteq Z(n, l, j),
\quad j=1, ..., J(n, l).$$ Then $$s((n+1, l, l', j'))=(n, l).$$ Note that there is 
$Z(n, l'')$ such that
$$Z(n+1, l', j'+J(n, l)+j)\subseteq Z(n, l'', j),\quad j=1, ..., J(n, l').$$ Then $$
(n+1, l, l', j')+1=(n+1, l'', l, j'+J(n, l)),$$ and
$$s((n+1, l, l', j')+1)=(n, l'').$$ If $$Z(n, l, J(n, l))\subseteq N_{n-1, s}\quad 
\textrm{and} \quad Z(n, l, J(n, l)+1)\subseteq N_{n-1, t}$$ for some $1\leq 
s, t\leq k$, then $$Z(n+1, l', j'+J(n+1))\subseteq N_{n-1, s}\quad
\textrm{and}\quad Z(n+1, l', j'+J(n+1)+1)\subseteq N_{n-1, t}.$$
But $$Z(n+1, l', j'+J(n+1)+1)=\sigma(Z(n+1, l', j'+J(n+1)+1)),$$ and this 
forces $s=t$, that is, 
$$m_-(s((n+1, l, l', j')+1))=m_-((n, l''))=t=s=m_+((n, l)).$$ This shows 
Condition \ref{prop-4-1} of Definition \ref{defn-BD}. 

Condition \ref{prop-4-2} of Definition \ref{defn-BD} can also be verified in a 
similar way.
\end{proof}

Applying the same argument as that of  \cite[Theorem 4.4]{HPS-Cantor},
 one can prove the following statement. 
 
\begin{thm}
Given $(X,\sigma)$ and the points $y_1\in Y_1,..., y_k\in Y_k$, the equivalent 
class of the ordered Bratteli diagram $B$ constructed in Theorem 
\ref{model-k-simple} does not depend on the choice of Kakutani-Rohklin 
partitions.
\end{thm}

\begin{thm}\label{model}
There is a one-to-one correspondence between the equivalence classes of non-elementary $k$-simple ordered Bratteli diagrams  and the pointed topological conjugacy 
classes of Cantor systems with $k$ minimal invariant subsets.
\end{thm}

\section{Transition graphs of a $k$-simple ordered Bratteli diagram and the index maps}\label{tran-graph}

\subsection{Transition graphs associated to Bratteli diagrams}

Consider a $k$-simple ordered Bratteli diagram $B=(V, E, >)$. Condition 
\ref{cont-cond} of Definition \ref{defn-BD} is necessary for the
 continuity of the 
associated Bratteli-Vershik map. Based on this continuity condition, we will 
introduce a sequence of directed graphs for $B$, and it turns out that these 
graphs are closely related to the index map of the short exact sequence 
\eqref{exact} associated to the Bratteli-Vershik system $(X_B, \sigma)$ (see 
Theorem \ref{index-diag}). Moreover, we will be able to obtain certain 
combinatorial properties of these graphs and the Bratteli diagram $B$ from 
the information on the index map.

In the next definition, we use the word ``vertex'' for a Bratteli diagram
and for  a transition graph. It should be clear from the context which vertex
is considered.

\begin{defn}\label{definition-TD}
Let $B=(V, E, >)$ be a $k$-simple ordered Bratteli diagram. For each level 
$n\geq 2$, define the {\em transition graph} $L_n$ to be the following 
directed graph: The vertices of $L_n$ correspond to the minimal sets
 $Y_1, ..., Y_k$, and the edges are 
labelled by the vertices from the set  $V_o^n$. For each $v\in V^n_o$, the
 edge $v$ starts from $Y_i$ and ends at $Y_j$ if and only if $$m_-(v)=i
 \quad \textrm{and}\quad m_+(v)=j.$$
\end{defn}
 
\begin{example}
For the diagram considered in Example \ref{ex-2-BD}, we obtain the
 transition graph at level $n$ as follows:

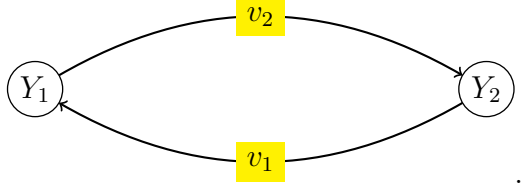
\begin{center}
\begin{tikzpicture}
\node[VertexStyle] (Y1) at (0, 0) {$Y_1$};
\node[VertexStyle] (Y2) at (6, 0) {$Y_2$};
\tikzstyle{EdgeStyle}=[bend left, ->]
 \Edge[label=$v_1$](Y2)(Y1)
  \Edge[label=$v_2$](Y1)(Y2)
 \end{tikzpicture}.
\end{center}

\end{example}

\begin{lem}\label{path-lifting}
Let $v\in V_o^n = V^n\setminus\bigcup_{i=1}^k V_i^n$ and let $e\in
 E_{\min}$ with 
$r(e)=v$ and $s(e)\in V^{n-1}_i$. Take $v'\in V^{n+1}$ and $e'\in E$ with 
$r(e')=v'$ and $s(e')=v$. Then, for any $e''\in E$ with $r(e'')=v'$ and 
$e'<e''$, there is a path $(v_1, ..., v_l)$ in $L_{n}$ with $v_1=v$ which 
starts at $Y_i$ and ends at $Y_{m_-(s(e''))}$.
\end{lem}
\begin{proof}
Since $e'<e''$, there are 
$$\{\{\underbrace{e_1^{(1)}, ..., e^{(1)}_{m_1}}_{G_1}\}, 
\{\underbrace{e^{(2)}_{1}, ..., e^{(2)}_{m_2}}_{G_2}\}, ...,
\{\underbrace{e^{(d)}_{1}, ..., e^{(d)}_{m_d}}_{G_{d}}\}, e'''\}$$
with each $e^{(s)}_j$ an edge between level $n-1$ and level $n$ such that
\begin{enumerate}
\item each group $G_s$ consists of all edges with range $v_s\in V^n$,
\item inside each group $G_s$, one has $$e_1^{(s)}<e_2^{(s)}(=e_1^{(s)}
+1)<\cdots<e_{m_s}^{(s)}(=e_{m_s-1}^{(s)}+1),$$
\item $e_1^{(1)}=e$, $e'''\in E_{\min}$ and $r(e''')=s(e'')$,
\item\label{lem-cond-004} there are edges $g_1, ..., g_s$ connecting $v_s$ to $v'$ respectively such that 
$$e'=g_1<g_2(=g_1+1)<g_3(=g_2+1)<\cdots<g_d(=g_{d-1}+1)<g_d
+1=e''.$$
\end{enumerate}

Then each group $G_s$ represents an edge in $L_n$ if $v_s\in V^n
\setminus \bigcup_{j=1}^k V_j^n$. Moreover, if $v_s\in V_j^n$ for some $j
$, by Condition \ref{prop-4-2} of Definition \ref{defn-BD}, one has that $m_-
(v_{s+1})=j$. Thus, if one deletes these $v_s\in V_j^n$ for some $j$, the 
remaining vertices induce a path in $L_n$. 

It is clear that the path starts with $Y_i$. The endpoint of the path is 
$Y_{m_+(v_d)}$. Then it follows from Condition \ref{prop-4-1} of Definition 
\ref{defn-BD} and Condition \ref{lem-cond-004} above that $$ Y_{m_+
(v_d)}=Y_{m_-(s(g_d+1))}=Y_{m_-(s(e''))},$$ as desired.

The argument can be illustrated by the following diagram

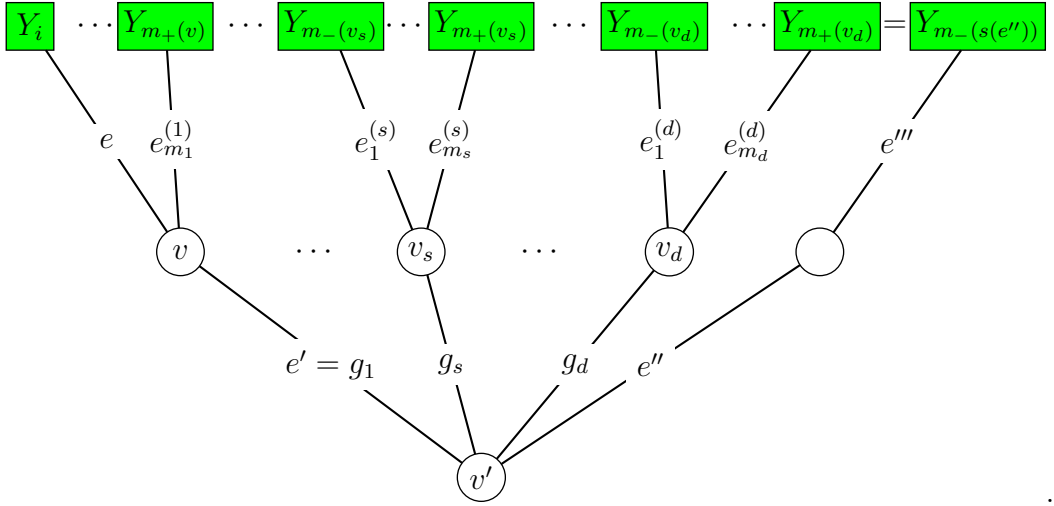
\begin{center}
\begin{tikzpicture}
\SetVertexNormal[Shape = rectangle, FillColor = green];
\node[VertexStyle] (Y1) at (-0.5, 6) {$Y_i$};
\node[VertexStyle] (E) at (12.1, 6) {$Y_{m_-(s(e''))}$};
 \node[VertexStyle] (G0) at (3.5, 6) {$Y_{m_-(v_s)}$};
  \node[VertexStyle] (G1) at (5.5, 6) {$Y_{m_+(v_s)}$};
  \node[VertexStyle] (Gd0) at (4.5+3.3, 6) {$Y_{m_-(v_d)}$};
  \node[VertexStyle] (Gd1) at (6.8+3.3, 6) {$Y_{m_+(v_d)}$};
  \node[VertexStyle] (G11) at (1.3, 6) {$Y_{m_+(v)}$};

\SetVertexNormal[Shape = circle];
\node[VertexStyle] (V) at (1.5, 3) {$v$};
\node[VertexStyle] (V') at (5.5, 0) {$v'$};
\node[VertexStyle] (V'') at (8, 3) {$v_d$};
\node[VertexStyle] (E0) at (10, 3) {};
\node[VertexStyle] (VS) at (4.7, 3) {$v_s$};

  \node at (11, 6) {$=$};
   \node at (3.3, 3) {$\cdots$};
  \node at (6.3, 3) {$\cdots$};

   \node at (4.5, 6) {$\cdots$};
  
  \node at (9.1, 6) {$\cdots$};
  \node at (6.7, 6) {$\cdots$};
    
\node at (0.4, 6) {$\cdots$};
\node at (2.4, 6) {$\cdots$};

\tikzstyle{EdgeStyle}=[-]

 \Edge[label=$e^{(s)}_1$](VS)(G0) 
  \Edge[label=$e^{(s)}_{m_s}$](VS)(G1)

\Edge[label=$e^{(d)}_1$](V'')(Gd0) 
  \Edge[label=$e^{(d)}_{m_d}$](V'')(Gd1)

\Edge[label=$e^{(1)}_{m_1}$](V)(G11)

 \Edge[label=$e$](V)(Y1)
 \Edge[label=${e'=g_1}$](V)(V')
 \Edge[label=$g_s$](V')(VS) 
  \Edge[label=$g_d$](V')(V'') 
  \Edge[label=${e''}$](V')(E0) 
  \Edge[label=$e'''$](E0)(E)

 \end{tikzpicture}.
\end{center}

\end{proof}

\begin{cor}\label{loop-diag}
Let $B=(V, E, >)$ be $k$-simple non-elementary ordered Bratteli diagram  
with $k\geq 2$, 
and let $L_n$ denote the transition graph of $B$ at level 
$n$. If there is an edge $v_1$ that has the vertex $Y_i$ as the source 
point, then there is a closed walk $(v_1, ..., v_l)$ in $L_n$ (so the range point of $v_l$ is $Y_i$).
\end{cor}

\begin{proof}
Without loss of generality, assume that the vertex $Y_1$ is a source point. 
Then there is $v\in V^{n}\setminus \bigcup_{i=1}^k V^{n}_i$ and $e\in 
E_{\min}$ such that $r(e)=v$ and $s(e)\in V^{n-1}_1$. Since the diagram is 
non-elementary, by Condition \ref{Cantorset} of Definition \ref{defn-unordered-BD}, there is $v'\in V^{n+1}$ such 
that $$\abs{E_{v, v'}:=\{e'\in E;\ s(e')=v, r(e')=v'\}}\geq 2.$$

Pick $e', e''\in E_{v, v'}$ with $e'<e''$. Then, by Lemma \ref{path-lifting}, 
there is a path in $L_{n}$ staring at $Y_1$ and ending at $Y_{m_-(v)}
=Y_1$; and it is the desired loop.
\end{proof}

\subsection{Index maps and transition graphs} In this subsection, we shall 
give a description of the images of the index maps (see Theorem \ref{index}) 
using the transition graphs.

Consider a $k$-simple non-elementary ordered Bratteli diagram $B=(V, E, >)$, and consider the associated Cantor system $(X_B, \sigma)$. Note that the ideal $\mathrm{C}_0(X_B\setminus\bigcup_{i=1}^k Y_i)\rtimes\Int$ has an AF-structure arising naturally from the sub-diagram of $B$ with vertices $\{V_o^n: n=1, 2, ...\}$.  In particular, its $\Kzero$-group is naturally isomorphic to the dimension group of the sub-diagram of $B$ with vertices $\{V_o^n: n=1, 2, ...\}$. With this identification, one has the following theorem:
\begin{thm}\label{index-diag}
Let $Y_i$ be a minimal component of $(X_B, \sigma)$, and let $L_n$ be the 
transition graph of $B$ at level $n$. Denote by $E_+(Y_i)$ the set of 
edges of $L_n$ which have $Y_i$ as source, and denote by $E_-(Y_i)$ 
the set of edges of $L_n$ which have $Y_i$ as range. Write 
$$E_+(Y_i)=\{v^+_1, ..., v^+_s\}$$ and $$E_-(Y_i)=\{v^-_1, ..., v^-_t\}.$$ 
That is,
 
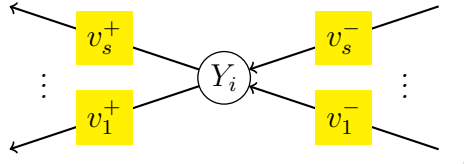
\begin{center}
\begin{tikzpicture}
\node[VertexStyle] (Y1) at (3, 1) {$Y_i$};
\node (E+1) at (0, 0) {};
\node (E+2) at (0, 2) {};

\node (E-1) at (6, 0) {};
\node (E-2) at (6, 2) {};

\node at (0.6, 1) {$\vdots$};
\node at (6-0.6, 1) {$\vdots$};

\tikzstyle{EdgeStyle}=[->]
 \Edge[label=$v_1^+$](Y1)(E+1)
  \Edge[label=$v_s^+$](Y1)(E+2)
  
   \Edge[label=$v_1^-$](E-1)(Y1)
  \Edge[label=$v_s^-$](E-2)(Y1)
 \end{tikzpicture}.
\end{center}
Let $d_i$ be the element of Theorem \ref{index} associated to $Y_i$. Then 
$d_i$ is given by $$(e_ {v^+_1}+\cdots+e_ {v^+_s})-(e_ {v^-_1}+\cdots
+e_ {v^-_t}),$$ where $e_ {v}$ stands for $(0, ..., 0, 1, 0, ..., 0)\in
\bigoplus_{V_o^n}\Int$ with entry $1$ at the position $v$.
\end{thm}

\begin{proof}
Consider the set $U\subseteq X_B$ consisting of all infinite paths which are in 
$Y_i$ up to level $n-1$.  It is clear that $U$ is a clopen set containing $Y_i$. 
By Theorem \ref{index}, the element $d_i$ is given by $$[\chi_{U}-\chi_{U}
\circ\sigma^{-1}]_0=[\chi_{U}-\chi_{\sigma(U)}]_0.$$ 

Denote by $v_1, ..., v_h$ the vertices in $V_o^n$ connecting to $Y_i$; and 
for each $1\leq\ l\leq h$, define
$$E_l:=\{e\in E;\ r(e)=v_l, s(e)\in V_i^{n-1}\};$$ and 
$$\{f^{(l)}_1, ..., f^{(l)}_{m_l}\}:=\{e\in E_l;\ e\notin E_{\max}, s(e+1)
\notin V^{n-1}_i\},$$
$$\{g^{(l)}_1, ..., g^{(l)}_{r_l}\}:=\{e\in E_l;\ e\notin E_{\min}, s(e-1)
\notin V^{n-1}_i\}.$$

Since the number of edges jumping into $E_l$ is the same as the number of 
edges being pushed out of $E_l$, one has
\begin{equation}\label{cancelation-k}
\left\{
\begin{array}{ll}
m_l=r_l,& \textrm{if $\abs{(E_l\cap E_{\max})\cup(E_l\cap E_{\min})}\in
\{0, 2\}$};\\
m_l+1=r_l, &  \textrm{if $(E_l\cap E_{\max})\neq\O$ but $E_l\cap E_{\min}
=\O$};\\
m_l=r_l+1, &  \textrm{if $(E_l\cap E_{\max})=\O$ but $E_l\cap E_{\min}
\neq\O$}.
\end{array}
\right.
\end{equation}

For each $e\in E_l$, define $\chi_-(e)$ to be the cylinder set consisting all 
infinite paths starting with $we$ with $w$ the minimal finite path ending at 
$s(e)$. Note that for any $e_1, e_2\in E_l$, one has 
\begin{equation}\label{same-K}
e_{v_l}:=[\chi_-(e_1)]_0=[\chi_-(e_2)]_0=(0, ..., 0, 1, 0...,0)\in
\bigoplus_{V_o^n}\Int\subset\Kzero(I),
\end{equation} 
where $1$ is at the position $v_l$.

One then asserts that 
\begin{equation}
\chi_{U}-\chi_{\sigma(U)}=\sum_{l=1}^h (\chi_-(g^{(l)}_1)+\cdots+\chi_-
(g^{(l)}_{r_l}))-\sum_{l=1}^h (\chi_-(f^{(l)}_1+1)+\cdots+\chi_-(f^{(l)}
_{m_l}+1)).
\end{equation}
In fact, $$\chi_{U}-\chi_{\sigma(U)}=\chi_{U\setminus\sigma(U)}-
\chi_{\sigma(U)\setminus U}.$$ Then the assertion follows from the 
equations
$$\chi_{U\setminus\sigma(U)}=\sum_{l=1}^h (\chi_-(g^{(l)}_1)+\cdots+
\chi_-(g^{(l)}_{r_l})),$$ and
$$\chi_{\sigma(U)\setminus U}=\sum_{l=1}^h (\chi_-(f^{(l)}_1+1)+\cdots
+\chi_-(f^{(l)}_{m_l}+1)),$$ which can be verified straightforwardly.

Then, by Equation \eqref{cancelation-k} and Equation \eqref{same-K}, one 
has
\begin{eqnarray*}
&&[\chi_{U}-\chi_{\sigma(U)}]_0\\
&= &\sum_{l=1}^h [(\chi_-(g^{(l)}_1)+\cdots+\chi_-(g^{(l)}_{r_l})]_0-\sum_{l=1}^h[\chi_-(f^{(l)}_1+1)+\cdots+\chi_-(f^{(l)}_{m_l}+1)]_0\\
&= &\sum_{l=1}^h [(\chi_-(g^{(l)}_1)]_0+\cdots+[\chi_-(g^{(l)}
_{r_l})]_0)-([\chi_-(f^{(l)}_1)]_0+\cdots+[\chi_-(f^{(l)}_{m_l})]_0)\\
&=& \sum_{l=1}^h (r_l-m_l)e_{v_l}
=(e_ {v^+_1}+\cdots+e_ {v^+_s})-(e_ {v^-_1}+\cdots+e_ {v^-_t}),
\end{eqnarray*}
as desired.
\end{proof}

\begin{cor}\label{connected-diag}
Let $B=(V, E, >)$ be a $k$-simple ordered Bratteli diagram with $k\geq 2$. 
Then each transition graph $L_n$ is connected. In particular, $L_n$ has 
at least $k-1$ edges. 
\end{cor}

\begin{proof}
If there were a proper connected component of $L_n$, say, consists of 
$\{Y_{n_1}, ..., Y_{n_s}\}$. Then, by Theorem \ref{index-diag}, 
one has $d_{n_1}+\cdots+d_{n_s}=0$.
But this contradicts to the conclusion of Theorem \ref{index}.
\end{proof}

In general, $k-1$ can be attained. For example, consider the following 
stationary diagram
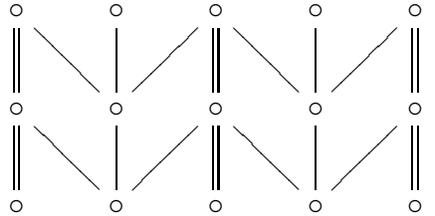
\begin{displaymath}
\xymatrix{
\circ \ar@{-}[dr] \ar@{=}[d]  & \ar@{-}[d] \circ & \circ \ar@{=}[d] \ar@{-}[dl] \ar@{-}[dr] & \circ \ar@{-}[d] & \circ \ar@{-}[dl] \ar@{=}[d] \\
\circ \ar@{-}[dr] \ar@{=}[d]  & \ar@{-}[d] \circ & \circ \ar@{=}[d] \ar@{-}[dl] \ar@{-}[dr] & \circ \ar@{-}[d] & \circ \ar@{-}[dl] \ar@{=}[d] \\
\circ  & \circ & \circ & \circ & \circ 
}.
\end{displaymath}
This diagram can be easily ordered so that it becomes an ordered $3$-simple 
Bratteli diagram. Its transition graph at each level has $2$ edges. 

However, if the Bratteli diagram is non-elementary, that is, if the path space is 
a Cantor set, then $L_n$ has at least $k$ edges.

\begin{cor}\label{k-vertices}
Let $B=(V, E, >)$ be a non-elementary $k$-simple ordered Bratteli 
diagram with $k\geq 2$. 
The transition graph $L_n$ has at least $k$ edges. 
In particular, one has that $$\abs{V_o^n}=\abs{V_n\setminus\bigcup_{i=1}
^k V_i^n}\geq k$$ for all $n$.
\end{cor}
\begin{proof}
It follows from Corollary \ref{connected-diag} that the transition graph $L_n$ 
is connected, and it follows from Corollary \ref{loop-diag} that $L_n$ 
contains loops. Since $L_n$ has $k$ vertices, it must has at least $k$ edges, 
as desired.
\end{proof}

\begin{defn}
Let $G$ be a dimension group with a given  inductive limit decomposition $G=\varinjlim \Int^{n_i}$. Define $\mathrm{D}(G)$ to be the subgroup consists of the elements $g$ such that, with $g=(g_i)$, where $g_i\in \Int^{n_i}$, there is $m\in\mathbb N$ such that $$\norm{g_i}_\infty \leq m,\quad i=1, 2, ... ,$$ where $\norm{\cdot}_\infty$ is the standard $\ell^\infty$-norm of $\Int^{n_i}$. 

It is straightforward to show that $\mathrm{D}(G)\cap G^+ = \{0\}$ if the dimension group $G$ has no quotient which is isomorphic to $\Int$. Also note that if $G$ is a noncyclic simple dimension group, then any element of $\mathrm{D}(G)$ is an infinitesimal (an element $g$ of a simple dimension group $G$ is said to be an infinitesimal if $-h<mg<h$ for all $m\in\mathbb N$ and $h\in G^+\setminus\{0\}$).
\end{defn}

Recall that $\mathrm{K}_{I_B}$ is the dimension group of the sub-diagram of $B$ restricted to the vertices $V_o^n$, $n=1, 2, ... ;$ hence it is isomorphic to $\Kzero(\mathrm{C}_0(X\setminus \bigcup_{i=1}^k Y_i)\rtimes\Int)$. 

\begin{cor}\label{index-inf}
If $B$ is a non-elementary ordered Bratteli diagram, then 
$$\mathrm{Image}(\mathrm{Ind})\subseteq \mathrm{D}({\mathrm{K}}_{I_B})$$
with respect to the inductive limit decomposition of $\mathrm{K}_{I_B}$ given by $B$. In particular, one has
$$\mathrm{Image}(\mathrm{Ind})\cap ({\mathrm{K}}_{I_B})^+=\{0\}.$$ 
Moreover, if $B$ is assumed to be strongly $k$-simple (so the ideal $\mathrm{K}_{I_B}$ is 
simple), then the image of the index map is in the subgroup of $\mathrm{K}_{I_B}$ of infinitesimals. 

\end{cor}
\begin{proof}
Since the image of the index map is generated by $\{d_1, ..., d_k\}$, it is enough to show that $$d_i \in \mathrm{D}(\mathrm{K}_{I_B}),\quad i=1, 2, ..., k.$$ But this follows directly from Theorem \ref{index-diag}, which states that each entry of $d_i$ must be $0$, $-1$, or $1$, at any level, as desired.
%
\end{proof}

The following can be regarded as a strengthened version of Corollary \ref{k-vertices}.
\begin{cor}\label{rank-of-ideal}
Denote by $r$ the $\Ratn$-rank of ${\mathrm{K}}_{I_B}$. Then $r\geq k$. 
\end{cor}
\begin{proof}
Denote by $H\subseteq {\mathrm{K}}_{I_B}$ the image of the index map. By Equation \ref{eq-ind}, one has that $\mathrm{dim}_\Ratn(H\otimes \Ratn) = k-1$. Note that the dimension group ${\mathrm{K}}_{I_B}$ must contain nonzero positive elements; so pick $p\in {\mathrm{K}}_{I_B}$ which is positive and nonzero. One asserts that $p\otimes 1_\Ratn\notin H\otimes\Ratn$. If this were not true, there are natural numbers $m, n$ and $h\in H$ such that $mp = nh$. Since $p$ is positive, it follows from Corollary \ref{index-inf} that $mp=0$. Since the dimension group ${\mathrm{K}}_{I_B}$ is torsion free, one has that $p=0$, which contradicts to the choice of $p$. Therefore ${\mathrm{K}}_{I_B}\otimes\Ratn \supsetneq H\otimes\Ratn$, and hence $r=\mathrm{dim}_\Ratn({\mathrm{K}}_{I_B}\otimes\Ratn) \geq k-1+1=k$. 
%
\end{proof}

The authors thank the referee for suggesting Pimsner's dynamical criterion for the stable finiteness in the following corollary.
\begin{cor}\label{2-0-fin}
Let $(X, \sigma)$ be a indecomposable Cantor system with $k$ minimal subsets.  
Then the C*-algebra $\mathrm{C}(X)\rtimes_{\sigma}\Int$ is stably finite. It has stable rank $2$ if $k\geq 2$, and stable rank $1$ if $k=1$.  Moreover, if $(X, \sigma)$ is aperiodic, the C*-algebra $\mathrm{C}(X)\rtimes_{\sigma}\Int$ has real rank zero. 
\end{cor}
\begin{proof}
Let us first consider the real rank of $\mathrm{C}(X)\rtimes_{\sigma}\Int$. Note that there is a short exact sequence
\begin{equation}\label{ext-cp-alg}
\xymatrix{
0\ar[r]&\mathrm{C}_0(X\setminus\bigcup_iY_i)
\rtimes_\sigma\Int\ar[r]&\mathrm{C}(X)\rtimes_\sigma\Int\ar[r]&
\bigoplus_i\mathrm{C}(Y_i)\rtimes_\sigma\Int\ar[r]&0
}.
\end{equation}
Since $(X, \sigma)$ is assumed to be aperiodic, each minimal component $Y_i$ is homeomorphic to a Cantor set, and therefore the quotient algebra $\bigoplus_i\mathrm{C}(Y_i)\rtimes_\sigma\Int$ is an A$\mathbb T$ algebra with real rank zero. On the other hand, the ideal $\mathrm{C}_0(X\setminus\bigcup_iY_i)\rtimes_\sigma\Int$ is AF, so it also has real rank zero. Since $\Kone({\mathrm{C}_0(X\setminus\bigcup_iY_i)\rtimes_\sigma\Int}) = \{0\}$, one has that the exponential map $$\Kzero(\bigoplus_i\mathrm{C}(Y_i)\rtimes_\sigma\Int) \to \Kone(\mathrm{C}_0(X\setminus\bigcup_iY_i)\rtimes_\sigma\Int) $$ of the extension above must be zero. Hence, as  an extension of two real rank zero C*-algebras with zero exponential map, the algebra $\mathrm{C}(X)\rtimes_\sigma\Int$ has real rank zero (see, for instance, Proposition 4 (i) of \cite{LM}).


For the stable rank of $\mathrm{C}(X)\rtimes_{\sigma}\Int$, if $k=1$, it follows from Corollary \ref{AT} that $\mathrm{C}(X)\rtimes_{\sigma}\Int$ is an A$\mathbb T$ algebra, and in particular, it has stable rank one. If $k\geq 2$, it follows from Corollary 2.6 of 
\cite{Poon-PAMS-89} that $\mathrm{C}(X)\rtimes_{\sigma}\Int$ has stable 
rank $2$. (Indeed, since both the ideal and the quotient algebras in the extension \eqref{ext-cp-alg} have stable rank $1$, it follows from Corollary 4.12 of \cite{Rieffel-DimStr} that the stable rank of $\mathrm{C}(X) \rtimes_\sigma\Int$ is either $1$ or $2$. Since the index map is nonzero when $k\geq 2$, the stable rank of $\mathrm{C}(X)\rtimes_{\sigma}\Int$ cannot be $1$; see, for instance, Proposition 4 (ii) of \cite{LM}.)


Let us show that $\mathrm{C}(X)\rtimes_{\sigma}\Int$ is always stably finite, and let us prove the following general statement instead: Consider an extension 
of C*-algebras 
\begin{displaymath}
\xymatrix{
0 \ar[r] & I \ar[r] & A \ar[r]^{\pi} & D \ar[r] & 0
}
\end{displaymath}
with $A$ and $D$ unital. Assume that $D$ is stably finite, $I$ has the 
property that $[p]_0\neq 0\in\Kzero(I)$ for any nonzero projection $p\in I$, 
and 
\begin{equation}\label{perp-pos}
\mathrm{Ind}(\Kone(D))\cap\Kzero^+(I)=\{0\}.
\end{equation} 
Then $A$ is stably finite. 

One only has to show that $A$ is finite (for matrix algebras over $A$, one 
can tensor the extension above with a matrix algebra, and proceed with the 
same argument). Let $v$ be an isometry in $A$. Since $D$ is finite,  the 
image $\pi(v)$ has to be an unitary. Then
$$\mathrm{Ind}(-[\pi(v)]_1)=[1_A-vv^*]_0-[1_A-v^*v]_0=[1_A-vv^*]_0\in\Kzero^+(I).$$
Therefore $[1-vv^*]_0=0$ and hence $1-vv^*=0$ (since $1-vv^*$ is a 
projection on $A$). So $v$ must be a unitary, and $A$ is finite.

Consider the extension \eqref{ext-cp-alg}. Note that, by Theorem \ref{model}, one may assume that the Cantor system arises from a non-elementary $k$-simple ordered Bratteli diagram. It follows from Corollary \ref{index-inf} 
that Equation \eqref{perp-pos} always holds, and $I$ is an AF-algebra. Hence the statement follows.

Alternatively, one also can use Pimsner's dynamical criterion for the stable finiteness. Note that the C*-algebra $\mathrm{C}(X)\rtimes_\sigma\Int$ is A$\mathbb T$ and hence stably finite if $k=1$ (Corollary \ref{AT}). So, let us assume that $k\geq 2$, and let us show that every point of $X$ is pseudoperiodic, i.e., for any $x_0\in X$ and any $\eps>0$, there exist $x_1, x_2, ..., x_{n-1}$ such that $$\mathrm{dist}(x_{i+1}, \sigma(x_i)) < \eps,\quad i=0, 1, ..., n-1,$$ where $\mathrm{dist}$ is a compatible metric on $X$, and $x_{n}$ is understood as  $x_0$. Then the stably finiteness follows from Theorem 9 of \cite{Pimsner-AF}.

The pseudoperiodicity indeed follows from Corollary \ref{connected-diag}: If $x_0 \in Y_i$ for sone $i=1, ..., k$, then the pseudoperiodicity follows from the minimality of $Y_i$; therefore one may assume that $x_0 \in X\setminus \bigcup_{i=1}^k Y_i$. With the Bratteli-Vershik model, let $x_0$ be represented by the infinite path $[e_1, e_1, ..., e_d, e_{d+1}, ...]$, and let $d$ be sufficiently large such that any two paths with same first $d$ segments actually have distance at most $\eps$. 

Consider the vertex $r(e_d)$, which is at the level $d+1$. Since $x_0\in X\setminus \bigcup_{i=1}^k Y_i$, one may assume that $d$ is sufficiently large so that $r(e_d) \in V_o^{d+1}$. Consider the minimal edge starting with $r(e_d)$ backwards, and denote it by $[e'_1, e'_2, ..., e_d']$. It is clear that $x_0$ is in the forward orbit of the infinite path $[e'_1, ..., e'_d, e_{d+1}, ...]$. By Lemma \ref{shorten}, one may assume that $s(e'_d)$ is in the sub-diagram $B_{m_{-}(r(e_d))}$ (and so the finite minimal path $[e'_1, e'_2, ..., e_{d-1}']$ is in the sub-diagram $B_{m_{-}(r(e_d))}$). Pick an arbitrary infinite path $y$ in the sub-diagram $B_{m_{-}(r(e_d))}$ which starts with $[e'_1, e'_2, ..., e_{d-1}']$. Note that $$\mathrm{dist}(y, [e'_1, ..., e_{d-1}', e'_d, e_{d+1}, ...])<\eps.$$

 Consider the minimal set $Y_{m_{-}(r(e_d))}$. By Corollary \ref{connected-diag}, there is a closed walk $(r(e_d), v_2, ..., v_l)$ in the transition graph $L_{d+1}$, where $v_i\in V_o^{d+1}$, $i=2, ..., l$, and $m_+(v_l) = m_-(r(e_d))$. Then, this loop provides a partial orbit $x_1, x_2, ..., x_n$, where each $x_i$ is an infinite path of the Bratteli diagram such that $x_{i+1}=\sigma(x_{i})$, $i=0, ..., n-1$, and $\sigma(x_n) = (e_1'', ..., e_{d-1}'', e_d'', ...)$, where $e_1'', e_2'', ..., e_{d-1}''$ are minimal edges in the Bratteli sub-diagram $B_{m_{-}(r(e_d))}$. Pick an arbitrary infinite path $z_0$ in the Bratteli sub-diagram $B_{m_{-}(r(e_d))}$ which starts with $[e_1'', e_2'', ..., e_{d-1}'']$, and note that $$\mathrm{dist}(z_0, \sigma(x_n)) = \mathrm{dist}(z, [e''_1, ..., e_{d-1}'', e''_d, ...])<\eps.$$

Since the set $Y_{m_{-}(r(e_d))}$ is minimal (so the forward orbit is dense), there are $z_1, z_2, ..., z_{n_1}\in Y_{m_{-}(r(e_d))}$ such that $z_{i+1} = \sigma(z_{i})$, $i=0, ..., n_1-1$, and $\mathrm{dist}(\sigma(z_{n_1}, y)) < \eps.$ Therefore, the finite sequence 
$$
x_0, x_1, x_2, ..., x_n, z_0, z_1, ..., z_{n_1}, y, [e'_1, ..., e_{d-1}', e'_d, e_{d+1}, ...], ..., x_0
$$
is the desired pseudoperiodic orbit.
\end{proof}

\section{Realizability of a Bratteli diagram}\label{EHS}

Let $B=(V, E)$ be an unordered $k$-simple Bratteli diagram. In this section, let us consider the question when there is an order $>$ on $B$  so that $(V, E, >)$ is an ordered $k$-simple Bratteli diagram.

Suppose that there is an order $>$ on $B$ so that $(V, E, >)$ is $k$-simple. Without loss of generality, let us also assume that $(V, E, >)$ satisfies the conditions of Lemma \ref{OBD-red}. 

Denote by $L_n$ the transition graphs of $(V, E, >)$ at the level $n$.  For any edge $v$ of the transition graph $L_n$, denote by $Y_{\min(v)}$ the source point of $v$ and denote by $Y_{\max(v)}$ the range point of $v$. 

\begin{thm}\label{realization-TD}
Consider the $k$-simple ordered Bratteli diagram $(V, E, >)$. The transition graphs $\{L_n;\ n=2, ...\}$ are compatible to the unordered Bratteli diagram $(V, E)$ in the following sense:
For any edge $w$ of $L_{n+1}$, there is a path $(v_1, v_2, ..., v_l)$ in $L_n$ such that 
\begin{enumerate}
\item\label{comp-01} the edge $w$ and the path $(v_1, ..., v_l)$ have the same range and source,
\item\label{comp-02} for any $v\in V^n_o$, the number of times $v$ (as an edge of $L_n$) appears in $(v_1, ..., v_l)$ is the same as the multiplicity of the edges in the Bratteli diagram $(V, E)$ between $v$ and $w$ (as vertices of $(V, E)$), 
\item\label{comp-03} if $w$ (as a vertex in $V_o^{n+1}$) is connected to some vertex in $V_i^n$ for some $1\leq i\leq k$, then $(v_1, v_2, ..., v_l)$ passes through $Y_i$, and
\item\label{comp-04} for any edge $v$ of $L_n$, the vertex $v$ (as a vertex in the Bratteli diagram) is connected to some vertex in $V_{\min(v)}^{n-1}$ and is also connected to some vertex in $V_{\max(v)}^{n-1}$.
\end{enumerate}

Conversely, if there is a sequence of directed graphs $\{L_n; n=2, 3, ...\}$ such that the vertices of each $L_n$ are $\{Y_1, ...,Y_k\}$, the edges of each $L_n$ are labelled by the vertices in $V_o^n$, and $(L_n)$ are compatible with $(V, E)$ in the sense above, then there is an order on $(V, E)$ so that it is a $k$-simple ordered Bratteli diagram.
\end{thm}
\begin{proof}
Let $w$ be any edge of $L_{n+1}$. With slight abusing of notation, let $w$ also denote the vertex in $V_o^{n+1}$ which corresponds to this directed edge of $L_{n+1}$. Since the edges $r^{-1}(w)\in E_{n}$ of the Bratteli diagram are totally ordered, write them as
$$e'_1<e'_2< \cdots< e'_m.$$ Remove all the edges with the source points not in $V_o^{n}$, and write the remaining edges as
$$e_1<e_2<\cdots < e_l.$$ Put $v_i=s(e_i), i=1, ..., l$. Then it is a direct calculation to show that $(v_1, v_2, ..., v_l)$ is a path in $L_n$ and satisfies Conditions \ref{comp-01},  \ref{comp-02}, \ref{comp-03} and \ref{comp-04}. We leave it to readers.

Now, assume that there are directed graphs $\{L_n;\ n=2, 3, ...\}$ and a $k$-simple (unordered) Bratteli diagram $(V, E)$ which are compatible. Let us show that there is an order $>$ on $(V, E)$ so that $(V, E, >)$ it is a $k$-simple ordered Bratteli diagram.

For each $1\leq i\leq k$ and $n\geq 1$, choose a pair of vertices $(z^n_{i, \min}, z^n_{i, \max})$ in $V^n_i$. (The infinite paths  $(z^1_{i, \min}, z^2_{i, \min}, ...)$ and $(z^1_{i, \max}, z^2_{i, \max}, ...)$ are going to be the minimal path and maximal path of the final ordered Bratteli diagram respectively.)

Then for each $v\in V_i^n$, put an arbitrary total order on $r^{-1}(v)$ such that $(z^{n-1}_{i, \min}, v)$ is minimal and $(z^{n-1}_{i, \max}, v)$ is maximal.

Now, let us consider how to order the edges $r^{-1}(w)$ for some $w\in V_o^n$.

On the first level, put an arbitrary total order on $r^{-1}(w)$ if $w\in V_o^1$. 

For any $w\in V_o^n$ with $n\geq 2$, define the order the edges $r^{-1}(w)$ as following:

Pick any edge between $V^{n-1}_{\min(w)}$ and $w$ to be the minimal edge and pick any edge between $V^{n-1}_{\max(w)}$ and $w$ to be the maximal edge. (The existence of such edges is insured by Condition \ref{comp-04}.)


If $n=2$, then order the edges $r^{-1}(w)$ with an arbitrary total order with the given minimal element and maximal element.

If $n>2$, then consider the corresponding path $(v_1, v_2, ..., v_l)$ in the transition graph $L_{n-1}$. For each $v_i$, $1\leq i\leq l$, pick an arbitrary edge in $E_n$ connecting $v_i$ (as a vertex in $V_o^{n-1}$) to $w$, and denote it by $e(v_i)$. By Condition \ref{comp-02}, such edge exists and the collection $$\{e(v_1), e(v_2), ..., e(v_l)\}$$ exhausts all the edges between $V_o^{n-1}$ and $w$. Set 
$$e(v_1)<e(v_2)< \cdots <e(v_l).$$ 

For each $1\leq i\leq k$, and edges between $V_i^{n-1}$ and $w$, by Condition \ref{comp-03}, there is $1\leq j\leq l$ such that the vertex $Y_i$ of $L_{n-1}$ is the range of $v_j$. Then insert all edges between $V_i^{n-1}$ and $w$ to between $e(v_j)$ and $e(v_{j+1})$ in an arbitrary order, and then one obtains a total order on $r^{-1}(w)$.

Note that it follows from the construction above that $$m_-(w)=Y_{\min(w)}\quad\textrm{and}\quad m_+(w)=Y_{\max(w)}.$$

Let us verify that the resulting ordered Bratteli diagram $(V, E, >)$ is $k$-simple.

Since $(V, E)$ is a $k$-simple unordered Bratteli diagram, Condition \ref{cond-unorder} of Definition \ref{defn-BD} is automatic.

By the choices of the ordering, it is easy to see that $(z^1_{i, \min}, z^2_{i, \min}, ...)$ and $(z^1_{i, \max}, z^2_{i, \max}, ...)$, $1\le i\leq k$, are the only maximal infinite mathes and minimal infinite paths. Thus Condition \ref{minimalsets} of Definition \ref{defn-BD} is also satisfied.

Let us verify Condition \ref{cont-cond} of Definition \ref{defn-BD}. 

Fix any $v\in V_o^n$. 

Let $e$ be any edge with $s(e)=v$. Denote by $w=r(e)$. Write the path in the transition graph $L_n$ corresponding to $w$ (as an edge in the transition graph $L_{n+1}$) to be $(v_1, v_2, ..., v_l)$. Then there is $1\leq i\leq l$ such that  $v=v_i$. By the construction of the transition graph, one has that $m_+(v)$ is the range of $v_i$ in $L_n$, that is
\begin{equation}\label{coc-01}
m_+(v)=Y_{\max(v_i)}.
\end{equation}

Consider the edge $e+1$. By the construction of the order on $r^{-1}(w)$, the vertex $s(e+1)$ is either $v_{i+1}$ or in $V^n_{\max(v_i)}$. If $s(e+1)=v_{i+1}$, then $$m_-(s(e+1))=m_-(v_{i+1})=Y_{\min(v_{i+1})}=Y_{\max(v_{i})};$$
and if $s(e+1)\in V^n_{\max(v_i)}$, then $m_-(s(e+1))=Y_{\max(v_{i})}$. Therefore,  by \eqref{coc-01}, one always has $$m_-(s(e+1))=Y_{\max(v_{i})}=m_+(v),$$ which verifies Condition \ref{prop-4-1} of Definition \ref{defn-BD}.

Now, let $e$ be an edge with $e\notin E_{\max}$, $r(e)=v$ and $s(e)\in V_i^{n-1}$ with $n\geq 3$. Write the path in $L_{n-1}$ corresponding to $v$ (as an edge in $L_n$) as $(u_1, u_2, ..., u_s)$. By Condition \ref{comp-03}, there is $1\leq j\leq s$ such that  $Y_{\max(u_j)}=Y_i$. 

If $j<s$, then one has that either $s(e+1)=u_{j+1}$, in which case, $$m_-(s(e+1))=Y_{\min(u_{j+1})}=Y_{\max(u_{j})}=Y_i$$ or $s(e+1)\in V_i^{n-1}$. So, in both cases, one has that $m_-(s(e+1))=Y_i$.

If $j=s$, since $e\notin E_{\max}$, one has that $s(e+1)\in V_i^{n-1}$, and therefore $m_-(s(e+1))=Y_i$.

Thus, the order satisfies Condition \ref{prop-4-2} of Definition \ref{defn-BD}, and hence $(V, E, >)$ is a $k$-simple ordered Bratteli diagram.
\end{proof}

Let $B=(V, E, >)$ be an ordered $k$-simple Bratteli diagram, and denote by $L_2, L_3, ...$ the corresponding transition graphs. Consider the Cantor system $(X_B, \sigma)$, and denote by $Y_1, ..., Y_k$ the minimal subsets. For each $1\leq i\leq k$, recall that $d_i$ is the image of $\chi_{Y_i}u$ under the index map. By Theorem \ref{index-diag}, 
each entry of $d_i\in\bigoplus_{V_o^n}\Int$ has to be $\pm 1$ or $0$. Moreover, one also has that for each $v\in V_o^n$, $$\abs{\{1\leq i\leq k;\ d_i(v)\neq 0\}}=0\ \mathrm{or}\ 2,$$ and if $d_{i_1}(v)\neq 0$ and $d_{i_2}(v)\neq 0$ for some $i_1\neq i_2$, then $d_{i_1}(v)+d_{i_2}(v)=0$; that is, if there is a nonzero pair, it must be either $(+1, -1)$ or $(-1, +1)$. 

%
%
%

Thus, the unordered Bratteli diagram $(V, E)$ has the following property:
there are elements $d_1, ..., d_k$ in ${\mathrm{K}}_{I_B}$ such that 
\begin{enumerate}
\item[(a)]\label{RC-01} $c_1d_1+\cdots+c_nd_n=0$ if and only if $c_1=c_2=\cdots=c_n$;
\item[(b)]\label{RC-02} for each level $n$ and each $v\in V_o^n$, one has that $d_i(v)\in\{0, \pm 1\}$, $1\leq i\leq k$;
\item[(c)]\label{RC-03} for each $v\in V_o^n$, one has that $\abs{\{1\leq i\leq k;\ d_i(v)\neq 0\}}=0\ \mathrm{or}\ 2,$ and if $$\{1\leq i\leq k;\ d_i(v)\neq 0\}=\{i_1, i_2\},$$ then $(d_{i_1}(v), d_{i_2}(v))$ is either $(+1, -1)$ or $(-1, +1)$;
%
%
%
\end{enumerate}
It turns out the these conditions are also sufficient for the existence of a $k$-simple order on $(V, E)$ if $(V, E)$ is strongly $k$-simple.

\begin{thm}\label{realization-IF}
Let $B=(V, E)$ be an unordered strongly $k$-simple Bratteli diagram satisfying Condition \ref{tel-cond-3} of Lemma \ref{OBD-red} (that is, each vertex in $V_o^{n+1}$ is connected to all vertices in $V^n$). Suppose that there are element $d_1, ..., d_k\in {\mathrm{K}}_{I_B}\subseteq  {\mathrm{K}}_B$ satisfying Conditions (a), (b), and (c) 
above. Then there is an order $>$ such that $(V, E, >)$ is an ordered (strongly) $k$-simple Bratteli diagram.
\end{thm}

Before we prove the theorem, let us recall several facts from graph theory.

\begin{defn}
Let $G=(V, E)$ be a directed graph (there might be multiple edges between two vertices, and loops are also allowed). Let $v$ be a vertex of $G$. The indegree and outdegree of $v$, denote by $\deg^-(v)$ and $\deg^+(v)$ respectively, are the numbers of directed edges leading into and leading away from $v$ respectively. The degree of $v$ is defined by $$\deg(v)=\deg^+(v)-\deg^-(v).$$

A (directed) Euler walk in $G$ is a walk (in the directed sense) in $G$ that covers each directed edge exactly once.
\end{defn}

The following is a criterion for the existence of an Euler walk in a directed graph. 

\begin{thm}\label{euler-wk}
A directed (multi)graph has an Euler walk if and only if it is connected, and $\deg(v)=0$ for every vertices with the possible exception of two vertices $v_0$ and $v_1$ such that $\deg(v_0)=1$ and $\deg(v_1)=-1$. In this case, $v_0$ and $v_1$ are the starting point and the end point of the Euler walk respectively.
\end{thm}

\begin{proof}[Proof of Theorem \ref{realization-IF}]
By Theorem \ref{realization-TD}, one only has to construct a sequence of directed graphs $L_2, L_3, ...$ which are compatible to $B=(V, E)$.

Note that the vertices of the proposed directed graphs are always $Y_1, Y_2, ..., Y_k$. To get $L_n$, one only has to assign each $v\in V_o^n$ to be a suitable directed edge of $L_n$. 

Fix $n\geq 2$. 

For each $v\in V_o^n$, if $d_i(v)=0$ for all $i$, then choose any $Y_i$ such that there is an edge between $V_i^{n-1}$ and $v$, and then assign $v$ to be a loop with the base point $Y_i$.

Otherwise, by Condition (c), 
there are $1\leq \min(v) \neq \max(v)\leq k$ such that $$d_{\min(v)}(v)=-1\quad\textrm{and}\quad d_{\max(v)}(v)=+1.$$ Then assign $v$ to be an edge from $Y_{\min(v)}$ to $Y_{\max(v)}$.

Denote by the resulting directed graph to be $L_n$, and one asserts that $\{L_2, L_3, ...\}$ are compatible to $(V, E)$ in the sense of Theorem \ref{realization-TD}.

First, note that for each $n\geq 2$ and $v\in V_o^n$, by the construction of $L_n$, one has 
\begin{equation}\label{index-recover}
d_i(v)=\left\{ 
\begin{array}{ll}
-1, & \textrm{if $Y_i$ is the source point but not range point of $v$ in $L_n$;}\\
+1, & \textrm{if $Y_i$ is the range point but not source point of $v$ in $L_n$;}\\
0, & \textrm{otherwise.}
\end{array}
\right.
\end{equation}
That is, the element $d_1, ..., d_k$ are induced by the diagram in the way of Theorem \ref{index-diag}.

Then it follows from Condition (a) 
that the underlining undirected graphs of $L_2, L_3, ...$ are connected. Indeed, if there were a proper connected component of $L_n$, say with vertices $Y_{n_1}, ..., Y_{n_t}$, it would then follow from \eqref{index-recover} that $d_{n_1}+\cdots+d_{n_t}=0$, which is in contradiction to Condition (a). 

Now, let $w$ be any edge of $L_{n+1}$. Consider $w$ as a vertex in $V_o^{n+1}$ and write $$w=\sum_{j=1}^l c_j v_j + \sum_{j=1}^{l'} c'_j v_j'$$ in the Bratteli diagram $(V, E)$, where $c_j, c'_j\in \mathbb N$, $v_j\in V_o^n$, and $v_j'\in V^n\setminus V_o^n$.

Since the vertex $w$ is assumed to be connected to all vertices at level $n$, one has that $\{v_1, ..., v_l\}=V_o^n$. Let us construct an auxiliary directed graph $L_n^w$ to be the directed (multi)graph obtained by multiple each edge $v_j$ of $L_n$ into $c_j$ edges. Since $L_n$ is connected, it is clear that $L_n^w$ is still connected.






Then, in order to find a path in $L_n$ satisfying Conditions \ref{comp-01} and \ref{comp-02} of Theorem \ref{realization-TD},  it is enough to find an Euler walk (i.e., a walk which cover each edge exactly once) in $L_n^w$ which has the starting point $Y_{\min(w)}$ and the ending point with $Y_{\max(w)}$. Moreover, since the edges of the graph $L_n^w$ exhaust all vertices in $V_o^n$, then Conditions \ref{comp-03} and \ref{comp-04} of Theorem \ref{realization-TD} are satisfied automatically.

Since $L_n^w$ is connected, by Theorem \ref{euler-wk}, it is enough to show that 
$$\deg_{L_n^w}(Y_{\min(w)})=+1,\quad\deg_{L_n^w}(Y_{\max(w)})=-1$$ and $$\deg_{L_n^w}(Y_i)=0\quad\textrm{for all other vertex $Y_i$}.$$ 

Consider any vertex $Y_i$ of $L_n^{w}$. Then 
$$\deg_{L_{n}^w}^+(Y_i)= \sum_{d_i(v_i)=-1} c_i\quad\textrm{and}\quad \deg_{L_{n}^w}^-(Y_i)= \sum_{d_i(v_i)=1} c_i,$$ and therefore
\begin{eqnarray*}
\deg_{L_{n}^w}(Y_i)&=& \deg_{L_{n}^w}^+(Y_i)-\deg_{L_{n}^w}^-(Y_i)\\
&=&\sum_{d_i(v_j)=-1} c_j -\sum_{d_i(v_j)=1} c_j \\
&=& -d_i(w).
\end{eqnarray*}
Hence,
\begin{displaymath}
\left\{
\begin{array}{l}
\deg_{L_{n}^w}(Y_i)=-d_i(w)=0,\quad  \textrm{if $i\notin\{\min(w), \max(w)\}$},\\
\deg_{L_{n}^w}(Y_{\min(w)})=-d_{\min(w)}(w)=1, \\
\deg_{L_{n}^w}(Y_{\max(w)})=-d_{\max(w)}(w)=-1. 
\end{array}
\right.
\end{displaymath}
Therefore, there is an Euler walk in $L_n^w$ with starting point $Y_{\min(w)}$ and ending point $Y_{\max(w)}$. That is, there is a path in $L_n$ satisfies Conditions \ref{comp-01}--\ref{comp-04} of Theorem \ref{realization-TD}, as desired.
\end{proof}

\begin{rem}
It would be interesting to know that for which (strongly) $k$-simple dimension group $G$, there is a (strongly) $k$-simple ordered Bratteli diagram $B=(V, E, >)$ such that $G$ is isomorphic to the dimension group of $(V, E)$. 
\end{rem}

\section{Cantor system with one minimal subset}\label{stable-alg}
Let us study the ordered $\Kzero$-group of $\mathrm C(X)\rtimes_\sigma\Int$ for a Cantor system with only one minimal component $Y$, and explore its connection to the boundedness of invariant measures on the open set $X\setminus Y$. By Theorem \ref{index}, the index map is zero. Also note that the C*-algebra $\mathrm{C}(Y)\rtimes_\sigma\Int$ is always an A$
\mathbb{T}$-algebra, and has real rank zero if $(X, \sigma)$ is aperiodic. Hence, by  \cite[Theorem 5]{LM}, we have the following structure theorem.
\begin{thm}
If $k=1$, the C*-algebra $\mathrm{C}(X)\rtimes_\sigma\Int$ is an 
A$\mathbb T$-algebra. It has real rank zero if $(X, \sigma)$ is aperiodic.
\end{thm}

Denote by $\mathcal K$ the algebra of compact operators acting on a separable infinite dimensional Hilbert space. Then a C*-algebra $A$ is said to be stable if $A \cong A\otimes\mathcal K$, and the positive cone of the $\Kzero$-group, denoted by $\Kzero^+(A)$, is defined by $$\Kzero^+(A) = \{[p]_0: \textrm{$p$ is a projection of $A\otimes\mathcal K$}\}\subseteq\Kzero(A).$$ Note that if $A$ is stably finite, $(\Kzero(A), \Kzero^+(A))$ is always an ordered group; furthermore, if $A$ is an A$\mathbb T$-algebras, $(\Kzero(A), \Kzero^+(A))$ is always a dimension group.

As in Section \ref{sect C.s. finitely many}, we have the following exact
 sequence
\begin{displaymath}
\xymatrix{
0\ar[r]&\Kzero(\mathrm{C}_0(X\setminus Y)\rtimes_\sigma\Int)\ar[r]^-
\iota&\Kzero(\mathrm{C}(X)\rtimes_\sigma\Int)\ar[r]^-\pi&
\Kzero(\mathrm{C}(Y)\rtimes_\sigma\Int)\ar[r]&0,
}
\end{displaymath}
and therefore $\Kzero(\mathrm{C}(X)\rtimes_\sigma\Int)$ is an extension 
of the dimension group. Denote by $u\in 
\Kzero^+(\mathrm{C}(X)\rtimes_\sigma\Int) $ and 
$v\in\Kzero^+(\mathrm{C}(Y)\rtimes_\sigma\Int)$  the 
standard order units induced by the constant function $1$. Then it is clear 
that $\pi(u)=v$. If, moreover, 
\begin{equation}\label{stb-cond}
\iota(\Kzero^+(\mathrm{C}_0(X\setminus Y)\rtimes_\sigma\Int))\subseteq [0, u],
\end{equation} 
then the extension above is an extension of the dimension group with 
order-units in the sense of \cite[page 295]{Goodearl}.  Since 
$\Kzero(\mathrm{C}(Y)\rtimes_\sigma\Int)$ is simple, it follows from 
 \cite[Theorem 17.9]{Goodearl} that the extension is lexicographic, namely, 
$$\Kzero^+(\mathrm{C}(X)\rtimes_\sigma\Int)=\iota(\Kzero^+
(\mathrm{C}_0(X\setminus Y)\rtimes_\sigma\Int))\cup \pi^{-1}(\Kzero^+
(\mathrm{C}(Y)\rtimes_\sigma\Int)\setminus\{0\}).$$

In general, relation \eqref{stb-cond} does not always hold. As we shall 
see in this section,  \eqref{stb-cond} holds if and only if there is no finite 
$\sigma$-invariant measure on the open set $X\setminus Y$, and hence 
the extension is lexicographic in this case.

\begin{example}\label{nonsimple}
The order-unit group $\Kzero(\mathrm{C}(Y)\rtimes_\sigma\Int)$ is always 
simple, but $\Kzero(\mathrm{C}_0(X\setminus Y)\rtimes_\sigma\Int)$ is 
not necessary a simple ordered group. For example, let $(X_1, \sigma_1)$ 
and $(X_2, \sigma_2)$ be two almost simple Cantor system with fixed point 
$x_1$ and $x_2$. Attaching $X_1$ to $X_2$ by identifying $x_1$ and 
$x_2$, we have a new Cantor set $X$ and an action $\sigma$ on it with a 
fixed point $\{x\}$. It is clear that the only nontrivial closed invariant subset 
is $\{x\}$ (although the system is not almost simple, which requires that 
every orbit other than $\{x\}$ is dense). However, $\Kzero(\mathrm{C}_0(X
\setminus \{x\})\rtimes_\sigma\Int)$ is not simple. One can expand the 
fixed point to an odometer system to get an aperiodic example. For instance, 
consider the stationary Bratteli diagram with incidence matrix 
\begin{displaymath}
F= \left(
\begin{array}{ccc}
2 & 2 & 0\\
0 & 2 & 0\\
0 & 2 & 3
\end{array}
\right).
\end{displaymath}
The dimension group associated with the ideal is $\Int[1/2]\oplus\Int[1/3]$ 
with the usual order.
\end{example}

%

Let $A$ be a stably finite C*-algebra. Denote by $$D(A)=\{[p]_0;\ p\in A\}
\subseteq\Kzero^+(A).$$
\begin{lem}\label{lemma-stable}
Let $A$ be a C*-algebra with an approximate unit consisting of projections. 
Assume that $A$ has stable rank one. Then $A$ is stable if and only if 
$D(A)=\Kzero^+(A)$.
\end{lem}
\begin{proof}
If $A$ is stable, then it is clear that $D(A)=\Kzero^+(A)$.

Assume that $D(A)=\Kzero^+(A)$, and let us show that $A$ is stable. Since $A$ has an approximate unit consisting of projections, by Theorem 3.1 of \cite{HjelmborgRordam1998}, it is enough to show that for any projection $p\in A$, there is a projection $q\in A$ such that $q$ is Murray-von Neumann equivalent to $p$ and $q\perp p$. Indeed, it follows from the stable rank one that $A$ has cancelation of projections. Together with $D(A)=\Kzero^+(A)$, one has that, for the given projection $p$, there is a projection $s\in A$ such that $s$ is  Murray-von Neumann equivalent to $p\oplus p$, which implies that there is a subprojection $q'\leq s$ which is Murray-von Neumann equivalent to $p$. Using the cancelation of projections again, one has that the compliment projection $s-q'$ is also Murray-von Neumann equivalent to $p$. Since $A$ has stable rank one, there is a unitary $u\in \tilde{A}$ such that $u^*q'u = p$ (i.e., $q'$ and $p$ are unitarily equivalent). Then $q=u^*(s-q')u$ is the desired projection.
\end{proof}

The following result deals with the case of AF-algebras.

\begin{lem}\label{lemma-stable-2}
Let $A$ be an AF-algebra. Then $A$ is stable if and only if any nonzero 
trace on $A$ is unbounded.
\end{lem}

\begin{proof}
If $A$ is stable, then any nonzero trace is unbounded.

If $A$ is not stable, we construct a nonzero bounded trace on $A$. Write 
$$A=\varinjlim(A_n,\phi_{n}),$$ where each $A_n=\bigoplus_{i=1}^{l_n}
\mathrm{M}_{m_{n, i}}(\Comp)$. Then there is a projection $p\in 
A$ (one may assume that $p\in A_1$) such that for any $n$, there exists 
$1\leq i\leq l_n$, such that $$2\cdot\mathrm{rank}(\pi_{n, i}\circ\phi_{1, n}
(p)) > m_{n, i};$$ as, otherwise, for 
any projection $p$, 
one can find a projection $q$ such that $p\perp q$ and $p\sim q$. Then, it 
follows from Theorem 3.1 of 
\cite{HjelmborgRordam1998} that $A$ is stable, which 
contradicts to the assumption.

Set $$\tau_n=\frac{1}{m_{n, i}}\mathrm{Tr}\circ\pi_{n, i}: A_n\to\Comp.
$$ It is clear that $\tau_n$ is a tracial state on $A_n$, and $
\tau_n(p)>1/2$. Extend $\tau_n$ to a linear functional on $A$ with norm 
one, and still denote it by $\tau_n$. Pick an accumulation point of $\{\tau_n;
\ n=1,...,\infty\}$, and denote it by $\tau$. It is clear that $\tau$ is a trace 
on $A$ with norm at most one. Moreover, $\tau_n(p)>1/2$, we have that $
\tau(p)\geq 1/2$, and thus $\tau$ is nonzero. Therefore $A$ has
 a nonzero bounded trace, as desired.
\end{proof}

\begin{thm}\label{inv-stable}
The restriction of $\sigma$ on $X\setminus Y$ has no finite invariant nonzero 
measure if and only if $\mathrm{C}_0(X\setminus Y)\rtimes_\sigma\Int$ is 
stable.
\end{thm}
\begin{proof}
If $(X\setminus Y, \sigma)$ has a finite invariant measure, then it induces a 
finite trace $\tau$ on $\mathrm{C}_0(X\setminus Y)\rtimes_\sigma\Int$, 
and hence it can not be stable.

On the other hand, if $\mathrm{C}_0(X\setminus Y)\rtimes_\sigma\Int$ is 
not stable, then by Lemma \ref{lemma-stable-2}, there is a bounded
 nonzero trace 
on $\mathrm{C}_0(X\setminus Y)\rtimes_\sigma\Int$. The restriction of this 
trace to $\mathrm{C}_0(X\setminus Y)$ induces a finite nonzero invariant 
measure.
\end{proof}

Denote by $u$ the standard order-unit of 
$\Kzero(\mathrm{C}(X)\rtimes_\sigma\Int)$, and consider the generating set 
$$\iota^{-1}[0, u] =\{\kappa \in \Kzero^+(\mathrm{C}_0(X\setminus Y)\rtimes_
\sigma\Int) : \iota(\kappa) < u\} \subseteq\Kzero^+(\mathrm{C}_0(X\setminus Y)\rtimes_
\sigma\Int).$$ 
The set $\iota^{-1}[0, u]$ is then a {\em generating interval} in sense that it 
is a convex upward-directed subset which generates the whole ordered 
group, see \cite[Lemma 17.8]{Goodearl}.

We have the following:
\begin{thm}\label{stable-gen}
The ideal $\mathrm{C}_0(X\setminus Y)\rtimes_\sigma\Int$ is stable if and 
only if $\iota^{-1}[0, u]=\Kzero^+(\mathrm{C}_0(X\setminus Y)\rtimes_
\sigma\Int).$
\end{thm}
\begin{proof}
Assume that $\iota^{-1}[0, u]=\Kzero^+(\mathrm{C}_0(X\setminus Y)
\rtimes_\sigma\Int)$; that is, for any positive element $a\in \Kzero^+
(\mathrm{C}_0(X\setminus Y)\rtimes_\sigma\Int)$, one has that $a<u$ in $
\Kzero^+(\mathrm{C}(X)\rtimes_\sigma\Int)$. Then, for any projection $p$ 
in a matrix algebra of $\mathrm{C}_0(X\setminus Y)\rtimes_\sigma\Int$, 
there is a partial isometry $v$ in a matrix algebra of $\mathrm{C}(X)
\rtimes_\sigma\Int$ such that $vv^*=p$ and $v^*v\leq 1$. In particular, 
$v^*pv\in\mathrm{C}(X)\rtimes_\sigma\Int$. Since $\mathrm{C}_0(X
\setminus Y)\rtimes_\sigma\Int$ is an ideal, one has that $v^*pv\in
\mathrm{C}_0(X\setminus Y)\rtimes_\sigma\Int$. Hence $[p]\in 
D(\mathrm{C}_0(X\setminus Y)\rtimes_\sigma\Int)$. By Lemma \ref{lemma-stable}, the ideal $\mathrm{C}_0(X\setminus Y)\rtimes_\sigma\Int$ is 
stable.

If the ideal $\mathrm{C}_0(X\setminus Y)\rtimes_\sigma\Int$ is stable, 
then, for any $a\in\Kzero^+(\mathrm{C}_0(X\setminus Y)\rtimes_\sigma
\Int)$, there is a projection $p\in\mathrm{C}_0(X\setminus Y)\rtimes_
\sigma\Int$ such that $[p]_0=a$. It is clear that $p< 1$ in $\mathrm{C}(X)
\rtimes_\sigma\Int$, and therefore $a=[p]_0<[1]_0=u$. Hence, 
$\iota^{-1}[0, 
u]=\Kzero^+(\mathrm{C}_0(X\setminus Y)\rtimes_\sigma\Int)$.
\end{proof}

\begin{cor} 
The restriction of $\sigma$ on $X\setminus Y$ has no nonzero finite invariant measure if and only if $\mathrm{C}_0(X\setminus Y)\rtimes_\sigma\Int$ is stable, if and only if $\iota^{-1}[0, u]=\Kzero^+(\mathrm{C}_0(X\setminus Y)\rtimes_\sigma\Int)$, and if and only if the extension is lexicographic.
\end{cor}

\section{Chain transitivity}\label{Chain}

\subsection{Topologies on the group of homeomorphisms} Let $X$ be a 
Cantor set, and let $H(X)$ denote the group of all homeomorphisms of a 
Cantor set $X$. Since all Cantor sets are homeomorphic, we do not need to 
specify a particular Cantor set while studying the group $H(X)$. 
In particular, the Cantor set can be represented as the path space of a non-
simple Bratteli diagram.

 We recall that by an \textit{aperiodic Cantor dynamical system} $(X, \sigma)$, 
 we mean a homeomorphism $\sigma$ of a Cantor set $X$ such that for any $x
 \in X$ the orbit $\mathrm{Orbit}_{\sigma}(x)= \{\sigma^i(x) : i \in \Z\}$ is 
 infinite.

The set  $\mathcal Ap$ of aperiodic homeomorphism  was studied in 
\cite{BezuglyiDooleyMedynets2005}, 
\cite{BezuglyiDooleyKwiatkowski2006}, 
\cite{Medynets2006}, \cite{Medynets2007} from various points 
of view. We recall here a few results that will be used below.

Fix a metric $d$ on $X$ compatible with the clopen topology on $X$. 
There are several natural topologies defined on $H(X)$, see 
\cite{BezuglyiDooleyKwiatkowski2006}, \cite{GlasnerWeiss2008}, 
\cite{Hochman2008}. The most popular one is the topology of uniform 
convergence, $\tau_w$, that turns $H(X)$ into a Polish group. This topology 
can be defined in several equivalent ways, for instance, by the metric
\begin{equation}\label{metric}
D(\psi_1, \psi_2) = \sup_{x\in X} d(\psi_1(x), \psi_2(x)) + \sup_{x\in X} 
d(\psi^{-1}_1(x), \psi^{-1}_2(x)), \ \ \ \psi_1, \psi_2 \in H(X).
\end{equation}
Equivalently, the topology $\tau_w$ is generated by the base of 
neighborhoods $\mathcal W = \{W(\psi ; E_1,..., E_n)\}$ where
$$
W(\psi ; E_1,..., E_n) = \{f \in H(X) : f(E_1) = \psi(E_1), ... , f(E_n) = 
\psi(E_n)\}.
$$
Here $ \psi \in H(X)$, and $E_1,..., E_n$ are any clopen sets. Without loss of 
generality, we can assume that $(E_1,..., E_n)$ forms a clopen partition of 
$X$.

If $D(\sigma_n, \sigma)  \to 0$, we say that $\sigma$ is approximated by a sequence 
of homeomorphisms $(\sigma_n)$. We first remark that any homeomorphism of 
$X$ is approximated by aperiodic homeomorphism.

\begin{lem}[\cite{BezuglyiDooleyKwiatkowski2006}]
\label{lem Ap dense in H(X)}
The set $\cal Ap$ is dense in $(H(X), \tau_w)$.
\end{lem}

We introduce the notations for some classes of homeomorphisms of a 
Cantor set $X$: $\mathcal Min$ denotes the set of all minimal 
homeomorphisms (a homeomorphism  $\sigma$ is minimal if every 
$\sigma$-orbit is dense in $X$); $\mathcal Tt$ denotes the set of all 
topologically transitive homeomorphisms ($\sigma$ is {\em topologically 
transitive} if there exists a dense orbit); $\cal Mov$ is the set of all moving 
homeomorphisms (a homeomorphism $\sigma \in H(X)$ is called 
\textit{moving} if, for any nontrivial clopen set $E\subset  X$
$$
\sigma(E) \setminus E \neq \O \ \ \ \mbox{and}\ \ \  E \setminus 
\sigma(E) \neq \O
$$

The notion of moving homeomorphisms was defined in  
\cite{BezuglyiDooleyKwiatkowski2006}.

\begin{lem}[\cite{BezuglyiDooleyKwiatkowski2006}]
\label{lem closure of moving} 
The set of moving homeomorphisms is $\tau_w$-closed. An aperiodic 
homeomorphism $\sigma$ is moving if and only if $\sigma $ can be approximated 
by a sequence of minimal homeomorphisms (or topologically transitive 
homeomorphisms), that is $\mathcal Mov = 
\ov{\mathcal Min}^{\tau_w}= \ov{\mathcal Tt}^{\tau_w}$.
\end{lem}

\subsection{Chain transitive homeomorphisms}

Let $(X, \sigma)$ be an aperiodic Cantor dynamical system. We now recall  
several  notions related to \textit{chains} in $X$ defined by $\sigma$. A finite 
set $\{x_0, x_1,..., x_n\}$, where $\sigma(x_i) = x_{i+1}, i =0,...,n-1$, is called 
a {{\em $\sigma$-chain} (or simply a {\em chain}). Given $\e >0$ and $x,y \in X
$, an {\em $\e$-chain} from $x$ to $y$ is a finite sequence $\{x_0, x_1,..., 
x_n\}$ such that $x_0 = x, x_n = y$ and $d(\sigma(x_i), x_{i+1}) < \e, i 
=0,...,n-1$. In symbols, an $\e$-chain from $x$ to $y$ will be denoted by $x 
\stackrel{\e}\rightsquigarrow y$.

Given an aperiodic Cantor system $(X,\sigma)$, it is said that $\sigma$ is  {\em 
chain transitive} if for any two points $x,y\in X$ there exists an $\e$-chain 
from $x$ to $y$. 

The following result is a new characterization of moving homeomorphisms  as  
chain transitive ones. 

\begin{thm}\label{thm moving=chain transitive}
 An aperiodic homeomorphism $\sigma$ of a Cantor set $X$ is chain transitive if 
 and only if $\sigma$ is moving.
\end{thm}

\begin{proof} ($\Longrightarrow$) If $\sigma$ is not moving then there is a 
nontrivial clopen set $E$ such that $\sigma(E) \subset E$, and $F= E \setminus 
\sigma(E)$ is a nonempty clopen subset (the case when $E \subset \sigma(E)$ is 
considered similarly). Let $x_0, y_0$ be any distinct points from $F$ with 
$d(x_0,y_0) = \delta >0$. We note that $\sigma^i(x) \in \sigma^i(E) \subset 
\sigma(E), i >0$ for any $x\in E$. Take $0 <\e < \min\{\delta, d(y_0, \sigma(E)), 
d(x_0, \sigma(E))\}$ (we denote the distance between closed subsets of $X$ 
by the same letter $d$; it will be clear from the context what metric space is 
meant).  Then there is no $\e$-chain from $x_0$ to $y_0$ and  from $y_0$ 
to $x_0$ as well. Hence, $\sigma$ is not chain transitive.

($\Longleftarrow$) Conversely, let $\sigma$ be a moving homeomorphism of a 
Cantor set $X$, and  let $\e$  be a positive number. Take a partition of $X$ 
into a finite collection of clopen sets $C(i)$ such that $\mbox{diam}(C(i)) < 
\e$ for any $i =1,..., N$. We are going to show that, for any  $x, y\in X$, 
there is a finite $\e$-chain for $\sigma$ from $x$ to $y$. Suppose $x\in C(i_0)
$ and $\sigma(x) \in C(i_1)$. Then any $z$ from $C(i_1)$ can be considered as 
the target of $\e$-chain $\{x,z\}$ of length 1. So, if $y\in C(i_1)$ we are 
done. If not, we consider $X \setminus B(1)$ where  $B(1) = C(i_0)$. 
Because $\sigma$ is moving, the set   $\sigma(B(1))$ intersects $X \setminus B(1)
$. Let
$$
B(2) = \bigcup_{j\in I_1} C(j)
$$
 where $j\in I_1$ if $C(j) \cap \sigma(B(1)) \neq \O$. If $z$ is a point 
 from $B(2)$, then there exists an $\e$-chain from $x$ to $z$ of length 2. 
 Indeed, if $z\in C(j), j \in I_1$, take the $\e$-chain $\{x, x_1, z\}$ where 
$x_1 \in B(1)$ such that $\sigma(x_1) \in C(j)$. If $y\in B(2)$, we are done. 
Otherwise, we apply the same argument to the set $X \setminus (B(1) \cup 
B(2))$. In view of compactness of $X$, this procedure terminates in a finite 
number of steps. This means there exists $B(k)$ such that $y \in B(k)$, that 
is the point $y$ is the final point of an $\e$-chain of length $k$.
\end{proof}

\begin{rem}\label{rem mov implies chain transitivity} We observe that the 
proof of Theorem \ref{thm moving=chain transitive} uses essentially the 
topological structure of Cantor sets, namely, the existence of  partitions of 
Cantor sets into clopen sets of arbitrary small diameter. Below we give 
another proof of the implication

\centerline{\textit{``moving'' $\Longrightarrow $ ``chain transitivity''} }

\noindent
that is based on Lemma  \ref{lem closure of moving}

\begin{proof} Suppose $\sigma\in H(X)$ is moving. By Lemma \ref{lem closure 
of moving}, for any fixed $\e >0$, there exists a minimal homeomorphism $f 
= f_\e$ such that $D(\sigma, f) < \e$. Take any two points $x, y \in X$. By 
minimality of $f$,  find the smallest positive $k$ such that $d(f^kx, y)  < \e
$. Consider the following set $I =\{x= x_0, x_1 = f(x), x_2 = f^2(x)),..., x_k 
= f^{k}(x)\}$. We claim that $I$ is an $\e$-chain for $\sigma$ from $x$ to $y$. 
Indeed,
$$
d(\sigma(x_i), x_{i+1}) = d(\sigma(f^i(x), f(f^i(x))   < \sup_{z\in X} d(\sigma(z), f(z)) 
< \e,\ i = 0,1,..., k-1,
$$
and $d(x_k, y) < \e$. Thus, $\sigma$ is chain transitive.
\end{proof}
\end{rem}

Let $\cal ChT$ be the set of all chain transitive homeomorphisms. The 
following Corollary follows from Theorem \ref{thm moving=chain transitive}.

\begin{cor}\label{cor chain trans - closed}
The set $\mathcal ChT$ of chain transitive homeomorphisms of a Cantor set 
is closed in the topology $\tau_w$ of uniform convergence.
\end{cor}

Since the proof of the fact that any moving homeomorphism is chain 
transitive, given in Remark \ref{rem mov implies chain transitivity} does not 
use explicitly the fact that $X$ is a Cantor set, we immediately deduce the 
following corollary.

 \begin{cor}\label{cor chain transitive} Suppose  $\sigma$ is a homeomorphism 
 of a compact metric space $(\Omega, d)$ which is approximated by minimal 
 homeomorphisms, i.e., $\lim_i D(\sigma, f_i) = 0$ where each $f_i$ is minimal. 
 Then $\sigma$ is chain transitive.
 \end{cor}

\subsection{Homeomorphisms with finite number of minimal sets}

In this subsection, we consider the case when an aperiodic homeomorphism 
$\sigma \in H(X)$ has a finite number of minimal sets, say $Y_1,..., Y_k$, 
that is each $Y_i$ is a closed $\sigma$-invariant set such that the orbit 
$\mathrm{Orbit}_\sigma(z)$ is dense in $Y_i$ for any $z\in Y_i$. For 
simplicity, we will call such homeomorphisms {\em $k$-minimal}.

It worth reminding that we deal with indecomposable homeomorphisms 
$\sigma$, that is every minimal set $Y$ for a $k$-minimal homeomorphism 
$\sigma$ is not open and has empty interior. 

Given a closed subset $C$ of $X$, we  say that an open set $V$ is an  {\em 
$\e$-neighborhood of $C$} if $V \supset C$ and $d(C, x) < \e$ for any $x
\in V$.

\begin{lem}\label{lem clopen nbhd} Given an aperiodic Cantor system $(X,
\sigma)$, let  $\{Y_\alpha\}$ be the collection of all minimal subsets for 
$\sigma$. Let $V$ be any clopen subset of $X$ such that $V \supset 
\bigcup_{\alpha} Y_\alpha$. Then
there exist positive integers $K_+$ and $K_-$ such that
\begin{equation}\label{eq K+ K-}
\bigcup_{n=0}^{K_+}\sigma^n (V) = X\ \ \mbox{and}\ \ 
\bigcup_{n=0}^{K_-}\sigma^{-n} (V) = X.
\end{equation}
Moreover, the same result is true when the above condition $V \supset 
\bigcup_{\alpha} Y_\alpha$ is replaced by the condition $V\cap Y_\alpha 
\neq \O$  for any $\alpha$.
\end{lem}

\begin{proof}  Let $(X,\sigma)$ be a Cantor aperiodic dynamical system. 
Consider the $\sigma$-invariant open set  $Y =\bigcup_{n\in \Z} \sigma^n 
(V)$. If the closed set $X \setminus Y$ were nonempty, then it would  
contain a minimal set for $\sigma$; this is impossible because all minimal sets 
are subsets of $Y$. Thus, $X = \bigcup_{n\in \Z} \sigma^n (V)$; hence, by 
compactness, the latter is a finite union. Applying  appropriate powers of $
\sigma$ to the above relation, we find some $K_+\in \N$ and $K_-\in \N$ 
such that (\ref{eq K+ K-}) holds. The other statement of the lemma is proved 
similarly.
\end{proof}

In the case of finitely many minimal subsets, we can reformulate Lemma 
\ref{lem clopen nbhd} in a more appropriate form. For this,
we first observe the following useful fact.  Let $V \supset Z$ be any  
neighborhood of a minimal set $Z$ for a homeomorphism $\sigma$. It is 
easily seen that $\sigma^n(V) \supset Z\ (n \in \Z)$ because $Z$ is $\sigma
$-invariant. Therefore, the following result can be straightforwardly deduced 
from  Lemma \ref{lem clopen nbhd}.

\begin{cor}\label{cor finite union} Let $(X, \sigma)$ be an aperiodic Cantor 
system with minimal sets  $Y_1,..., Y_k$. For any $\e> 0$ and any clopen 
$\e$-neighborhood $V_i$ of $Y_i\ (i=1,..., k)$, the clopen set $V = 
\bigcup_{i=1}^k V_i$  satisfies the condition  $\bigcup_{n= 0}^{K_+}
\sigma^n (V) = X$, $\bigcup_{n= 0}^{K_-}\sigma^{-n} (V) = X$ for some 
positive integers $K_+$ and $K_-$.
\end{cor}

\begin{defn}\label{def equivalent wrt Z}
 Let   $Z$  be a  fixed minimal set for an aperiodic Cantor system 
 $(X,\sigma)$.  We will say that two points $x,y \in X$ are \textit{chain 
 equivalent with respect to $Z$} if for any $\e >0$ there exist $\e$-chains 
 $x \stackrel{\e}\rightsquigarrow z_0$ and $y \stackrel{\e}\rightsquigarrow 
 z_0$ where  $z_0$ is a point from $Z$. 
\end{defn}

We observe that if $x \stackrel{\e}\rightsquigarrow z_0$ and $y 
\stackrel{\e}\rightsquigarrow z_0$ for some point $z_0\in Z$, then $x 
\stackrel{\e}\rightsquigarrow z$ and $y \stackrel{\e}\rightsquigarrow z$ for 
any point $z$ because $\sigma$ is minimal on $Z$.

\begin{lem}\label{lem TFAE chain wrt Z} Let $Z$ be a minimal set for an 
aperiodic Cantor system $(X,\sigma)$, and let $x,y$ be any two  points from 
$X$. The following statements are equivalent:

(i) the points $x$ and $y$ are chain equivalent with respect to $Z$;

(ii) $\forall \e >0\ \exists\ x \stackrel{\e}\rightsquigarrow z_1 \ \exists \ y 
\stackrel{\e}\rightsquigarrow z_2$ where $z_1, z_2 \in Z$;

(iii) $\forall \e >0\ \mbox{and}\ \forall \e\mbox{-neighborhood}\ V_\e \ 
\exists\ x \stackrel{\e}\rightsquigarrow v_1 \ \exists\ y \stackrel{\e}
\rightsquigarrow v_2$ where $v_1, v_2 \in V_\e$.
\end{lem}

It follows from this lemma that the chain equivalence with respect to a 
minimal set $Z$ is an equivalence  relation on $X\times X$. We denote it by $
\mathcal E(Z)$.

\begin{proof} We sketch the proof of the lemma because the technique used 
in the proof  is rather standard. Clearly, the implications $(i) \Longrightarrow  
(ii)  \Longrightarrow (iii)$ are obvious, so that it remains  to show that  $(iii) 
\Longrightarrow  (i)$. To do this, it suffices to notice that, by  minimality of $
\sigma$ on $Z$,  there exists a point $z\in Z$ (in fact, $z$ can be any point 
from $Z$) such that  $v_1 \stackrel{\e}\rightsquigarrow z$ and  $v_2 
\stackrel{\e}\rightsquigarrow z$. The result then follows Definition
\ref{def equivalent wrt Z}.
\end{proof}

\begin{prop}\label{prop unique minimal set} Suppose $Z$ is a unique minimal 
set for an aperiodic homeomorphism $\sigma$ of a Cantor set $X$. Then 
$\mathcal E(Z) = X\times X$. In other words, any two points in $X$ are 
chain equivalent with respect to $Z$.
Moreover, $\sigma$ is  chain transitive.
\end{prop}

\begin{proof}  The fact that $\mathcal E(Z) = X\times X$ follows directly 
from Lemmas \ref{lem clopen nbhd} and \ref{lem TFAE chain wrt Z} because 
the $\sigma$-orbit of any clopen neighborhood of $Z$ covers $X$ after 
finitely many iterations.

Take any two points  $x,y$  in  $X$.
Let $\e > 0$, and let $V_{\e}$ be an $\e$-neighborhood of $Z$.  
We consider the case when $x,y \in X \setminus Z$; the other possible cases 
are considered similarly with obvious simplifications. By Corollary \ref{cor 
finite union}, we find positive integers $K_+$ and $K_-$ such that
$$
\bigcup_{n=0}^{K_+}\sigma^n (V_\e) =  \bigcup_{n=0}^{K_-}\sigma^{-n} 
(V_\e) = X.
$$
Let $i$ be the smallest number from the set $(0, 1,..., K_-)$ such that $x \in 
\sigma^{-i}(V_\e)$, and let $j$ be the smallest number from the set $(0, 
1,..., K_+)$ such that $y \in \sigma^{j}(V_\e)$. To construct a chain $C$ 
from $x$ to $y$, we find first two  points $w\in Z$ and $z\in Z$ such that 
$d(w, \sigma^{-j}(y)) < \e$ and  $d(\sigma^i(x), z) < \e$. Then we take the 
set
$$
C = (x, \sigma(x),..., \sigma^i(x), z, \sigma(z), ..., \sigma^m(z), \sigma^{-j}
(y), \sigma^{-j +1}(y),..., y),
$$
where  $m$ is the smallest positive integer such that $d(\sigma^m(z), w) < 
\e$. Then $C$ is a $2\e$-chain for $\sigma$. This proves the chain 
transitivity  of $\sigma$.
\end{proof}

\begin{rem} (1) It follows from Theorem \ref{thm moving=chain transitive} 
that if $(X,\sigma)$ is an aperiodic Cantor system with a unique  minimal set, 
then $\sigma$ is a moving homeomorphism.

(2) Another simple observation from the proved result consists of the 
following fact: any aperiodic homeomorphism with a unique minimal set is the 
limit of a sequence of minimal homeomorphisms in the topology of uniform 
convergence $\tau_w$.
\end{rem}

Dynamical properties of an aperiodic Cantor  system $(X, \sigma)$ become  
more complicated when the system has several minimal sets $Y_1,..., Y_k$.

 Let $E_i$  be the subset of $X$ defined as follows:
\begin{equation}\label{eq E_i}
E_i = \{x \in X : \forall \e >0 \ \exists\ x \stackrel{\e}\rightsquigarrow z \ 
\mbox{for\ some}\  z\in Y_i\},\ i =1,..., k.
\end{equation}
We observe that $E_i$ contains $Y_i$, and moreover, $E_i$ is the $\mathcal 
E(Y_i)$-saturation of the set $Y_i$. The latter means that $E_i$ is the 
minimal $\mathcal E(Y_i)$-invariant set such that every $\cal E(Y_i)$-orbit 
meets $Z$ at least once.

\begin{lem}\label{lem E_i is clopen} In the above notation, each $E_i$ is a 
clopen subset of $X,\ i=1,..., k$. 
\end{lem}

\begin{proof} First we show that $E_i$ is open. For $x \in E_i$ and $\e > 0$, 
take an $\e$-chain $(x= x_0, x_1,..., x_n=z) $ from $x$ to $z \in /Y_i$. We 
prove that there exists a neighborhood $V$ of $x$ such that $V \subset E_i
$. Choose $\delta >0$ such that $d(\sigma(x), \sigma(y)) < \e/2$ when $y 
\in V := \{u : d(x,u)< \delta\}$. Then for any $y\in V$, we see that $(y, x_1, 
... , x_n)$ is an $\e$-chain from $y$ to $z$.

Next, let $(u(n))$ be a sequence from $E_i$ that converges to some point 
$x$. Suppose that for each $u(n)$ there is an $\e$-chain $(u(n), x_1(n),..., 
z(n))$ from $u(n)$ to $Y_i$. Choose $n_0$ such that $d(\sigma(u(n_0), 
\sigma(x)) < \e/2$. Then $(x, x_1(n),..., z(n))$ is an $\e$-chain from $x$ to 
$Y_i$. Hence, $E_i$ is closed.

\end{proof}

Consider an aperiodic Cantor system $(X,\sigma)$ with finitely many minimal 
sets $Y_1,..., Y_k$, and let $E_1, ... , E_k$ be the clopen sets defined by 
minimal sets according to (\ref{eq E_i}). Then, by Proposition \ref{prop 
unique minimal set}, we have
$$
X = \bigcup_{i=1}^k E_i.
$$

\begin{lem}\label{lem E_i  intersects and E_j}
In the above notation, 
$$
E_i \cap E_j \neq \O \ \ \Longrightarrow \ \ E_i = E_j,\ \ i,j = 1,...,k.
$$ 
\end{lem}

\begin{proof}
It follows from Proposition \ref{prop unique minimal set} that, for any $\e 
>0$, each set $E_i$, and any two points $x,y \in E_i$, there exists an $\e$-
chain from $x$ to $y$. Thus, if $x_i \in E_i$,  $x_j \in E_j$  are arbitrary 
points and $y\in E_i \cap E_j$, then there are $\e$-chains $x_i \stackrel{\e}
\rightsquigarrow z_j$ and  $x_j \stackrel{\e}\rightsquigarrow z_i$ where 
$z_i \in Y_i$ and $z_j\in Y_j$.
\end{proof}

\begin{thm}\label{thm chain transitive}
Suppose $\sigma$ is an aperiodic homeomorphism, and $Y_1, ... , Y_k$ are 
minimal sets for $\sigma$ on a Cantor set $X$. Then  $(X,\sigma)$ is chain 
transitive if and only if $E_1 = ... = E_k = X$.

\end{thm}

\begin{proof} Let $\sigma$ be a chain transitive homeomorphism. Fix $E_i$ 
and take any point $x \in E_i$. Then, for any $\e >0$ and   $y_j \in E_j$, 
there exists an $\e$-chain from $x$ to $y_j$, $j = 1, ... , k$. Hence, $x\in 
E_j$ for all $j = 1,...,k$. This proves the first statement.

The converse statement follows from Lemma 
\ref{lem E_i  intersects and E_j}.
\end{proof}

The properties of homeomorphisms with finitely many minimal sets are clearly 
seen for Vershik maps defined on $k$-minimal Bratteli diagrams. We obtain 
 the following result  from Theorem \ref{thm chain transitive}.

\begin{cor}\label{cor k-minimal homeo}
Let $B = (V, E, >)$ be a $k$-simple ordered Bratteli diagram and let
$(X_B, \omega_B)$ be a Bratteli-Vershik $k$-minimal 
dynamical system defined on the path space of $B$. 
Then $\sigma_B$ is a chain transitive homeomorphism.
\end{cor}

\end{document}